\documentclass{article}
\usepackage{graphicx} 
\usepackage{xcolor}
\usepackage{amsmath}
\usepackage{amsfonts}
\usepackage{subfigure}
\usepackage{caption}
\usepackage{url}
\usepackage{hyperref}
\usepackage{amsthm}
\newtheorem{theorem}{Theorem}[section]
\newtheorem{lemma}{Lemma}[section]
\newtheorem{rmk}{Remark}[section]
\newtheorem{corollary}{Corollary}[section]
\newtheorem{definition}{Definition}[section]
\newtheorem{example}{Example}[section]
\def\strutdepth{\dp\strutbox}
\def \BVR#1{\strut\vadjust{\kern-\strutdepth\vtop to0pt{\vss\hbox to\hsize {\hskip\hsize\hskip5pt$\leftarrow$\hss\strut}}}{\em \textcolor{blue}{#1}}}

\usepackage[
backend=biber,
style=numeric,
sorting=none
]{biblatex}
\title{Global Dynamics of Trait-Structured Generalised Lotka-Volterra Systems with Trait-Independent Interactions}
\author{
  Nataliya Balabanova\\
\texttt{nbalabanova@lincoln.ac.uk}
  \and
  Manh Hong Duong\\
  \texttt{h.duong@bham.ac.uk}
\and
 Blaine Van Rensburg\\
\texttt{blaine.vanrensburg@abdn.ac.uk}
}
\date{July 2026}

\begin{document}

\maketitle
\begin{abstract}
We study the long time dynamics of a selection-mutation integro-differential Lotka-Volterra system of $N$ populations. In our model, fitness depends on a continuous phenotypic trait, but the effect of one population on another is independent of this trait. We establish that, under some usual assumptions on the interactions between populations, the long time behaviour of solutions is exactly determined by the well studied Generalised Lotka-Volterra (ordinary differential) Equation. We also show that, all else being equal, the minimal mutation rate is selected for in a static environment, whereas in an a changing environment, an intermediate mutation rate could be selected instead. The key step in our proofs is establishing that the total population sizes are asymptotically governed by a Generalised Lotka-Volterra Equation, which then allows the use of a general result on asymptotically autonomous dynamical systems.  
\end{abstract}
\section{Introduction}

\subsection{Motivation}

Understanding the coexistence of species is a central problem in mathematical biology and theoretical ecology, concerned with explaining how multiple species persist despite competition for limited resources. Early work by Alfred J. Lotka and Vito Volterra \cite{lotka1925elements,volterra1926fluctuations} introduced dynamical systems approaches to interspecific interactions, laying the foundation for the modern theory of ecological competition and predation. These ideas were sharpened by the competitive exclusion principle associated with G. F. Gause, which showed that species competing for identical resources cannot stably coexist in homogeneous environments \cite{gause1934struggle}. This classical perspective led to extensive mathematical investigation of mechanisms that relax exclusion, including resource partitioning, spatial structure, and temporal environmental variability. Subsequent developments in structured and spatial ecology, notably in the work of Robert MacArthur and Richard Levins, demonstrated how heterogeneity in resources, habitats, and dispersal can generate coexistence in systems that would otherwise be unstable under mean-field assumptions \cite{macarthur1967limiting}. Modern coexistence theory, particularly the framework developed by Peter Chesson, unifies these perspectives by distinguishing stabilizing mechanisms that promote negative frequency dependence from equalizing mechanisms that reduce fitness differences among species, providing a general mathematical language for biodiversity maintenance \cite{chesson1985coexistence,chesson2000mechanisms}.

In parallel with these ecological developments, mathematical biology has increasingly emphasized the role of intra-specific variation in shaping population and community dynamics. Trait-based and structured population models show that ignoring phenotypic heterogeneity can lead to qualitatively incorrect predictions about stability, resilience, and coexistence. In particular, continuous trait variation can generate emergent niche structure, alter effective interaction strengths, and couple ecological and evolutionary timescales through mutation and selection \cite{des2018ecological}. These ideas connect to broader frameworks in adaptive dynamics and eco-evolutionary theory, where feedbacks between trait evolution and ecological interactions can either promote or inhibit coexistence depending on the structure of fitness landscapes and competition kernels.

Motivated by these developments, in this work, we investigate the role of intra-specific trait variation in species coexistence using an integro-differential equation framework. We consider $N$ interacting populations under three main assumptions: (1) each population is structured by an inheritable continuous trait which, together with a temporally inhomogeneous environment, determines fitness via a trait-dependent intrinsic growth rate; (2) reproduction is subject to mutation producing offspring with nearby trait values, inducing diffusion in trait space; and (3) interspecific interactions are pairwise and depend only on total population densities, representing mean-field competition. Within this setting, we aim to characterize how trait-mediated heterogeneity influences the conditions for long-term coexistence across multiple interacting species.

More precisely, we consider the co-evolutionary dynamics of $N$ species whose  population trait-density vector $(u_1(x_1,t),...,u_N(x_N,t))$  is governed by a system of the form
\begin{equation}\label{eqn:General}
      \begin{cases}
          \partial_{t}u_i=d_i\Delta{u}_i+u_i\left(r_i(x_i,t)-\sum_{j=1}^Na_{ij}\rho_{j}(t)\right),&(x_i,t)\in\Omega_i\times\mathbb{R}^{+},
    \\
    
    \partial_{\nu}u_i=0,&(x_i,t)\in\partial{\Omega_i}\times\mathbb{R}^{+},\\
    \rho_i(t)=\int_{\Omega}u_i(x,t)dx,&t\in\mathbb{R}^{+},\\
    
    u_i(x,0)=u_{i,0}(x),&x_i\in\Omega_i,
      \end{cases}
  \end{equation}
  where  $i\in\{1,...,N\}$. For the $i$-th species, $\Omega_i\subset\mathbb{R}^n$ is its trait-domain (which is either connected and bounded with smooth boundary, or is the whole space), $d_i>0$ is its mutation rate, $a_{ii}>0$ is the strength of interspecies competition, and $r_i(x_i,t)\in{}C^{2,1}(\Omega_{i}\times\mathbb{R}^{+})$ is the intrinsic birth rate which depends on external environmental factors. For the initial data, we assume that $0\leq{}u_{i,0}(x)\leq{}Ae^{-B|x|}$, and that none of the $u_{i,0}$ are identically $0$. The interactions between populations are determined entirely by the coefficients of the matrix $A=(a_{ij})_{i,j\in\{1,...,N\}}$.  
  
The system \eqref{eqn:General} belongs to a broad class of integro-differential selection--mutation models that provide a general framework for studying the survival and adaptation of populations undergoing Darwinian evolution. Models of this type arise in a variety of applications, including the prediction of extinction driven by environmental change \cite{roques2020adaptation,iglesias2021selection,adaptivedynamicsdivergingfitness}, the study of drug resistance in cancer cell populations \cite{busse2016mass}, and the description of co-evolutionary dynamics in interacting cell populations \cite{nguyen2019adaptive}. Various mathematical properties of selection--mutation models have been extensively studied \cite{barles2009concentration,desvillettes2008selection,alfaro2014explicit,lorenzi2020asymptotic}. Related multi-species advection-reaction--diffusion systems, involving diffusion in physical space rather than in trait space, have also attracted considerable attention; see, for example, \cite{dockery1998evolution,potts2019spatial,giunta2022local,giunta2022detecting,duong2025multi}. A similar model except without mutation is studied in \cite{pouchol2018global}, where a  sufficient conditions on the interaction matrix $A$ are obtained such that all species coexist. To the best of our knowledge, the specific model \eqref{eqn:General} that takes into account both mutation and selection mechanisms {combined with Lotka-Volterra dynamics} has not yet been studied in the literature.

\subsection{Summary of main results}

The aim of this paper is to provide a rigorous analysis of the system \eqref{eqn:General}, which poses significant challenges due to its integro-differential nature and nonlocal structure. Our key observation is that, by integrating \eqref{eqn:General}, we obtain the following equation for the evolution of the total population mass $\rho(t)=(\rho_1(t),...,\rho_N(t))$:
  \begin{equation}\label{eqn:GLV_P}
    \begin{cases}
        \frac{d\rho_i}{dt}=\rho_i\left(\bar{r}_{i}(t)-\sum_{i=1}^{N}a_{ij}\rho_j\right),\\
        \rho_i(0)=\int_{\Omega_i}u_{i,0}(y)\, dy,
    \end{cases}
\end{equation}
where $\bar{r}_i(t)=\int_{\Omega_i}\frac{u_i(x,t)}{\rho_i(t)}r_i(x_i,t)$ represents the average fitness of species $i$ at time $t$. 

The system \eqref{eqn:GLV_P} has the same form as the celebrated generalised Lotka-Volterra Equation,
\begin{equation}\label{eqn:GLV_ODE_generic}
    \frac{d\rho_i}{dt}=\rho_i\left(\bar{r}_{i,\infty}-\sum_{i=1}^Na_{ij}\rho_j\right),~~i=1,...,N,
\end{equation}
but with time-dependent reproduction rates (although note that for each $i\in\{1,...,N\}$, the function $\bar{r}_i(t)$ depends implicitly on $u_i$). If the convergence $\bar{r}_i(t)\xrightarrow[t\rightarrow\infty]{}\bar{r}_{i,\infty}$ holds for some constant $\bar{r}_{i,\infty}$, as is the case under some general assumptions on $A$, then we write $\bar{r}_i(t)=\bar{r}_{i,\infty}+\alpha_i(t)$ so that $\alpha_i(t)\xrightarrow[t\rightarrow\infty]{}0$ and the system \eqref{eqn:GLV_P} can be viewed as a perturbation of \eqref{eqn:GLV_ODE_generic}. 

Under appropriate conditions on $A$, we are able to control the perturbation $\alpha(t)=(\alpha_1(t),...,\alpha_N(t))$ and therefore determine the global dynamics of two important classes of submodels of \eqref{eqn:General} in terms of the global dynamics of \eqref{eqn:GLV_ODE_generic}. The first submodel (see system \eqref{eqn:MutationEvolution} below) is for bounded trait domains and time-independent birth rates, while the second one (see system \eqref{eqn:ShiftingEnvironment} below) is for one-dimensional traits, where the optimal trait for each species shifts due to a changing environment.

Our main results are Theorem \ref{thm:Global_attractors_BD} and Theorem \ref{thm:Global_attractors_UBD} which provide conditions under which the asymptotic behaviour (e.g., eventual coexistence, extinction of populations) of solutions to \eqref{eqn:General} is the same as the asymptotic behaviour of solutions to \eqref{eqn:GLV_ODE_generic}, respectively for the aforementioned two classes of submodels. More precisely, we show that if any of the following assumptions holds
\begin{enumerate}
    \item[(a)] there is a coexistence state of \eqref{eqn:GLV_ODE_generic} which is stable in a particular sense (to be defined in the next section),
    \item[(b)] interactions between the populations are either all competitive, or all mutualistic, and there is a unique coexistence state,
    \item[(c)] interactions are competitive and uniform in the sense that for a given population, all other populations interact with the same strength,
\end{enumerate}
then the fate of the interacting trait-structured populations is the same as that for a population without trait-structure and mutations, but with an intrinsic birth rate which encodes those properties. In other words, the long time dynamics of solutions to \eqref{eqn:General} are the same as those for a solution to \eqref{eqn:GLV_ODE_generic} given some usual assumptions on  the interaction matrix $A$. A corollary of the result in the case of uniform competition (iii) is that, when all else is equal, the minimal mutation rate is selected for in a temporally constant environment, while if the changing environment causes the optimal traits to shift sufficiently fast, an intermediate mutation rate is selected instead.

In addition, in Lemma \ref{lem: dynamics LV} we further study trajectorial properties of the non-autonomous generalised Lotka-Volterra equation \eqref{eqn:GLV_ODE_generic}. Furthermore, we also demonstrate analytical results by numerical simulations.   
\subsection{Organization of the paper}
In Section \ref{sec:ModelsStatements} we introduce the specific models (of the form \eqref{eqn:General}), present  the precise statements of Theorems \ref{thm:Global_attractors_BD} and \ref{thm:Global_attractors_UBD}, and provide additional remarks on their biological interpretation, before comparing our result to related models. Then, in Section \ref{sec:Proofs} we provide the detailed proofs of the main results. A delay effect that arises in generic perturbed systems is investigated in Section \ref{sec:RemarksDynamics}, and illustrated numerically. 

\section{Models and statement of results}\label{sec:ModelsStatements}
In this section, we introduce in details the two classes of models that we consider and state the main results obtained for them.
\subsection{Time-independent intrinsic birth rates}
We first consider the case where each of the intrinsic birth rate is independent of $t$, $r_i(x_i,t)=r_i(x_i)$, and the trait-domain $\Omega_i$ is a smooth and bounded for $i\in\{1,...,N\}$. In this case, the model reads
\begin{equation}\label{eqn:MutationEvolution}
      \begin{cases}
          \partial_{t}u_i=d_i\Delta{u}_i+u_i\left(r_i(x_i)-\sum_{j=1}^Na_{ij}\rho_{j}(t)\right),&(x_i,t)\in\Omega_i\times\mathbb{R}^{+},
    \\
    
    \partial_{\nu}u_i=0,&(x_i,t)\in\partial{\Omega_i}\times\mathbb{R}^{+},\\
    \rho_i(t)=\int_{\Omega}u_i(x,t)dx,&t\in\mathbb{R}^{+},\\
    
    u_i(x,0)=u_{i,0}(x),&x_i\in\Omega_i.
      \end{cases}
  \end{equation}
We assume that there is one interior global maximum for each $r_i$ and that $r_{i,M}=\max_{x\in\Omega_i}r_i(x)>0$ (since otherwise the species goes extinct). We also suppose that the instinct birth rate $r_i(x)$ is defined for $x\in\mathbb{R}^n$ since it aids the interpretation of our results later. Our main theorem relates the global dynamics of \eqref{eqn:MutationEvolution} to those of the generalised Lotka-Volterra system
\begin{equation}\label{eqn:GLV_1}
    \frac{d}{dt}\hat{\rho}_i^ {}=\hat{\rho}_{i}^{}\left(-\lambda_i-\sum_{j=1}^{N}a_{ij}\hat{\rho}_j^{}\right),~~\text{for }~i\in\{1,...,N\},
\end{equation}
for some suitable $\lambda_i\in \mathbb{R}$. To state the main result, we first carry out some formal computations in order to introduce relevant systems, which we will use to prove the existence of a unique classical solution to \eqref{eqn:MutationEvolution}. We define the transformed variables \begin{equation}\label{eqn:DecouplingEqn}
    \tilde{u}_i(x,t)=u_i(x,t)e^{\int_{0}^ t\sum_{j=1}^Na_{ij}\rho_j(s)ds+\lambda_it},
\end{equation}
where, for $i=1,\ldots, N$, the parameter $\lambda_i$ is the principal eigenvalue corresponding to the principal eigenfunction $p_i\in{}C^{2,1}_{\text{loc}}(\Omega_i)$ which satisfies
\begin{equation}\label{eqn:eigenfunctionsDomain}
      \begin{cases}
        -d_i\Delta{}p_i=p_i\left(r_i(x)+\lambda_i\right),&x\in{\Omega_i},
    \\
    \partial_\nu{}p_i=0,&x\in\partial{\Omega}_i,\\
   p_i>0,&x\in\overline{\Omega}_{i},\\
   \int_{\Omega_i}p_i(x)dx=1.
      \end{cases}
  \end{equation}

If $u_i(x,t)$ is a classical solution to \eqref{eqn:MutationEvolution}, then by differentiating $\tilde{u}(x,t)=(\tilde{u}_1(x,t),...,\tilde{u}_N(x,t))$ with respect to $t$ we obtain
\begin{align*}
\partial_{t}\tilde{u}_i&=\partial_{t}u_i\Big(e^{\int_{0}^ t\sum_{j=1}^Na_{ij}\rho_j(s)ds+\lambda_it}\Big)+\Big(\sum_{j=1}^Na_{ij}\rho_j(s)ds+\lambda_i\Big)u_ie^{{\int_{0}^ t\sum_{j=1}^Na_{ij}\rho_j(s)ds+\lambda_it}}\\
&=d_i\Delta{}\tilde{u}_i+\tilde{u}_i\Big(r_i(x_i)-\sum_{j=1}^Na_{ij}\rho_j(t)\Big)+\tilde{u}_i\Big(\sum_{j=1}^Na_{ij}\rho_j(t)+\lambda_i\Big)\\
    &=d_i\Delta{\tilde{u}_i}+\tilde{u}_i(r(x_i)+\lambda_i)
\end{align*}
so that $\tilde{u}_i$ is a classical solution to the following \textit{linear and decoupled} system 
\begin{equation}\label{eqn:Decoupled}
      \begin{cases}
          \partial_{t}\tilde{u}_i=d_i\Delta{\tilde{u}}_i+\tilde{u}_i\left(r_i(x_i)+\lambda_i\right),&(x_i,t)\in\Omega_i\times\mathbb{R}^{+},
    \\
    
    \partial_{\nu}\tilde{u}_i=0,&(x_i,t)\in\partial{\Omega_i}\times\mathbb{R}^{+},\\
    
    u_i(x,0)=u_{i,0}(x),&x_i\in\Omega_i.
      \end{cases}
  \end{equation}
We note that the eigenvalue problem \eqref{eqn:eigenfunctionsDomain} that determines the eigenfunction-eigenvalue pair $(p_i,\lambda_i)$ is precisely the stationary equation of the parabolic equation \eqref{eqn:Decoupled}. Their existence  follows from the theory for elliptic linear operators (reviewed in \cite{lam2022introduction}, Section 1.3), while existence of a classical solution to \eqref{eqn:Decoupled} follows from the theory for linear parabolic equations (see e.g., \cite{evans2022partial}, Chapter 7). The existence of a unique solution $u$ to \eqref{eqn:MutationEvolution} such that $u_i\in{}C^{2,1}(\Omega\times\mathbb{R}_{>0})\cap{}C^{1,0}(\bar{\Omega}\times[0,\infty))$ (for $i\in\{1,...,N\}$) is obtained next.

Integration of \eqref{eqn:DecouplingEqn} with respect to $x$ yields
\begin{align}\label{eqn:MassRelation}
    \tilde{\rho}_i(t)=\rho_i(t)e^{\int_{0}^t\sum_{j=1}^Na_{ij}\rho_j(s)ds+\lambda_it},~t\in[0,\infty),
\end{align}
where $\tilde{\rho}_j=\int_{\Omega_i}\tilde{u}_j(x,t)dx$, and $\tilde{u}=(\tilde{u}_1,...,\tilde{u}_N)$ is the solution to \eqref{eqn:Decoupled}. We claim that there is a unique solution $\rho$ to \eqref{eqn:MassRelation} defined for $t\in\mathbb{R}_{>0}$, such that $\rho_i\in{}C^1(\mathbb{R}_{>0})\cap{}C^0([0,\infty))$. In this case, we define
$u_i(x,t):=\tilde{u}_i(x,t)e^{-\int_{0}^ta_{ij}\rho_j(s)ds-\lambda_it}$ so that $u=(u_1,...,u_N)$ is the unique (classical) solution to \eqref{eqn:MutationEvolution}. The next lemma makes this precise.
\begin{lemma}\label{lma:MassEquivalent}
Assume that any solution $\hat{\rho}$ to 
\eqref{eqn:GLV_1}
with non-negative initial condition exists for all time $t\in\mathbb{R}_{>0}$. Also, assume that
\begin{itemize}
   \item[(G)] {The solution to}
        \begin{align*}
        \frac{d}{dt}\rho_i(t)=\rho_i(t)\left(-\lambda_i-\sum_{j=1}^Na_{ij}\rho_j(t)+\alpha_i(t)\right),
    \end{align*} 
     {is bounded for all $t\geq{0}$, where $\alpha_i(t):=-\frac{d\tilde{\rho}_i(t)}{dt}\frac{1}{\tilde{\rho}_i(t)}$, $\tilde{u}=(\tilde{u}_1,...,\tilde{u}_N)$ is the solution to \eqref{eqn:Decoupled}, and $\tilde{\rho}_i=\int_{\Omega_i}\tilde{u}_i(x,t)dx$.}
\end{itemize}
Then there is a unique solution $\rho=(\rho_1,...,\rho_N)$ which satisfies \eqref{eqn:MassRelation} with $\rho_i\in{}C^1(\mathbb{R}_{>0})$. Consequently, there is a unique solution $u=(u_1,...,u_N)$ to \eqref{eqn:MutationEvolution} where  $u_i\in{}C^{2,1}(\Omega\times\mathbb{R}_{>0})\cap{}C^{1,0}(\bar{\Omega}\times[0,\infty))$.
\end{lemma}
The proof of Lemma \ref{lma:MassEquivalent} will be given in Section \ref{sec: proof21}. In the next remark we comment on assumption (G) used in Lemma \ref{lma:MassEquivalent}.
\begin{rmk}
    The assumption (G) is included only to ensure global existence of a solution. We will verify that (G) is satisfied in each of the cases we consider in Theorem \ref{thm:Global_attractors_BD}. Note that it is \textit{not} in general satisfied even under the natural condition that $a_{ii}>0$ for each $i\in\{1,...,N\}$. Consider, for instance, the system (where $N=2$)
    \begin{align*}
        \frac{d\rho_1}{dt}&=\rho_1(2\rho_2-\rho_1)\\
        \frac{d\rho_2}{dt}&=\rho_2(2\rho_1-\rho_2),
    \end{align*}
    with initial condition $(\rho_1,\rho_2)=(\nu,\nu)$ for $\nu>0$. By symmetry we have that $\rho_1=\rho_2$ on the interval on which solutions exist, but in this case then $\frac{d\rho_1}{dt}=\rho_1^2$ and this solution blows up at $t=\frac{1}{\nu}<\infty$.
\end{rmk}
Let $A=(a_{ij})_{i,j=1}^N$ be a given matrix, and let ${\Lambda}=(\lambda_1,...,\lambda_N)$. Then we have 
the following next lemma, which determines the possible long-term behaviour assuming no species goes extinct or grows without bound.
\begin{lemma}\label{lma:MassBoundSolution}
\hspace{0.1em}
\begin{enumerate}
    \item If $\rho_i(t)$ is bounded above for all $t\geq{0}$ then $\limsup_{t>0}\frac{1}{t}\int_{0}^t\sum_{j=1}^Na_{ij}\rho_j(s)ds=-\lambda_i$ or $\rho_i(t)\xrightarrow[t\rightarrow\infty]{}0$.
    \item If $\rho_i(t)$ is uniformly bounded away from $0$ with respect to $t$, then either $\liminf_{t>0}\frac{1}{t}\int_{0}^t\sum_{j=1}^Na_{ij}\rho_j(s)ds=-\lambda_i$ or $\rho_i$ is unbounded.
    
    \item If $\rho_i(t)$ is bounded, and uniformly bounded away from $0$ with respect to $t$, then $\frac{1}{t}\int_{0}^ t\rho(s)ds\xrightarrow[t\rightarrow\infty]{}\rho^{*}\in\mathbb{R}^N$ which solves the equation \[{A}\rho^{*}=-{\Lambda}.\]
\end{enumerate}
\end{lemma}
The proof of Lemma \ref{lma:MassBoundSolution} will be presented in Section \ref{sec: proof21}. For the Generalised Lotka-Volterra model \eqref{eqn:GLV_1} trajectories may be unbounded or species may go extinct depending on the properties of the matrix $A$ and the initial condition. Consequently, we expect the same properties for the perturbed system \eqref{eqn:PopDynFull_1} unless we specify the properties of the matrix $A$ further. Therefore, to make use of Lemma \ref{lma:MassBoundSolution} we restrict to the following situations:
\begin{itemize}
    \item[a)] Lyapunov diagonal stability: there exists a positive diagonal matrix $P$ such that $P^TA+AP^T$ is positive definite. 
    \item[b)] Purely competitive (respectively, cooperative) interactions: $a_{ij}>0$ for $i,j\in\{1,...,N\}$  (respectively, $a_{ij}\leq{0}$ for $i\neq{}j\in\{1,...,N\}$).
    \item[c)] Uniform competitive interactions ( $a_{ij}=a_{i}>0$ for $i,j=\{1,...,N\}$).
\end{itemize}
The first main result of the paper is the following theorem that characterizes the long-time behaviour of the total population mass $\rho(t)=(\rho_1(t),\ldots, \rho_N(t))$ of the system \eqref{eqn:MutationEvolution}.
\begin{theorem}\label{thm:Global_attractors_BD}
Let $u=(u_1,...,u_N)$ be the solution to \eqref{eqn:MutationEvolution} and $\rho_i(x)=\int_{\Omega_i}u_i(x,t)dx$ for $i\in\{1,...,N\}$.

Suppose that:  $\rho^{*}=(\rho_1^*,\ldots, \rho_N^*)$ where $\rho_i^*>0$ for all $i=1,\ldots, N$, is the unique non-trivial equilibrium solution to \eqref{eqn:GLV_1}, $\lambda_i<0$ for all $i=1,\ldots, N$, and one of
\begin{itemize}
    \item[a)]  $A$ is Lyapunov diagonally stable, 
    \item[b)]  $a_{ij}>0$ for all $i,j\in\{1,...,N\}$, or $a_{ij}\leq{}0$ for $i\neq{j}\in\{1,...,N\}$.
\end{itemize}
holds. Then $\rho(t)\xrightarrow[t\rightarrow\infty]{}\rho^{*}$ and $\Vert{}\frac{u_i(\cdot,t)}{\rho_i(t)}-p_i\Vert_{L^\infty(\Omega_i)}\xrightarrow[t\rightarrow\infty]{}0$.

Alternatively, suppose only that

\begin{itemize}
    \item[c)] $a_{ij}=a_i>0$,  for each $i,j\in\{1,...,N\}$, and that there is a unique index $\underline{i}\in\{1,...,N\}$ such that $\frac{\lambda_{\underline{i}}}{a_{\underline{i}}}=\min_{i\in\{1,...,N\}}\frac{\lambda_i}{a_i}<0$.
\end{itemize} 
Then
   \begin{align*}
       &\left\Vert{u_j}\right\Vert_{L^\infty(\Omega_j)}\xrightarrow[t\rightarrow\infty]{}0,~~\text{for }i\neq{}\underline{i}\\
       &\rho_{\underline{i}}(t)\xrightarrow[t\rightarrow\infty]{}-\frac{\lambda_{\underline{i}}}{a_i},~~ \left\Vert{}\frac{u_{\underline{i}}(\cdot,t)}{\rho_{\underline{i}}(t)}-p_i\right\Vert_{L^\infty(\Omega_{\underline{i}})}\xrightarrow[t\rightarrow\infty]{}0.
   \end{align*}

\end{theorem}
The proof of this theorem will be given in Section \ref{sec: proof21}. For each case, once we have shown condition (G) (which says that $\rho$ is bounded, and is assumed in Lemma \ref{lma:MassEquivalent}), we may use general result for asymptotically autonomous ODEs with bounded solutions (which is given in detail in \ref{sec: proof21}) to show that $\rho$ has the same long time dynamics as solutions to its autonomous counterpart \eqref{eqn:GLV_1}. Indeed, this is exactly what we do for case $b)$ (where $a_{ij}>0$ for $i,j\in\{1,...,N\}$), requiring no further assumptions on $A$. The benefit of this approach is that it does not necessarily require the construction of a Lyapunov function. However, in the cases where a Lyapunov function is available, it is used to show (G), at which point it's possible to conclude either directly using the Lyapunov function, or by using the general result for asymptotically autonomous ODEs. 

\begin{rmk}
Theorem \ref{thm:Global_attractors_BD} can be extended to the following class of systems of elliptic equations 
\begin{align*}
    \begin{cases}
        \partial_tu_i=\nabla\cdot(D_i(x_i)\nabla{}u_i)+u_i\left(r_i(x)-\sum_{j=1}^Na_{ij}\rho_j\right),&(x_i,t)\in{\Omega_i}\times\mathbb{R}^+,\\
        \partial_{v}u_i=0,&(x_i,t)\in\partial\Omega_i\times\mathbb{R}^+,\\
        \rho_i(t)=\int_{\Omega_i}u_i(x,t)dx,&t\in\mathbb{R}^+.
    \end{cases}
\end{align*}
where diffusion matrices $D_i(x_i)=(d_{jk}^i(x_i))$ (for $i\in\{1,...,N\}$) are  anisotropic and dependent on trait, provided that each $D_i$ is uniformly elliptic and continuous in $x$, and the domain $\Omega_i$ is connected with sufficiently smooth boundary. In this case, the only step of the proof affected is obtaining the principal eigenvalue-eigenfunction pairs, which would be obtained from the following equations instead 
\begin{equation*}\label{eqn:eigenfunctionsDomainAnistropic}
      \begin{cases}
        -\nabla\cdot(D_i\nabla{}u_i)=p_i\left(r_i(x)+\lambda_i\right),&x\in{\Omega_i},
    \\
    \partial_\nu{}p_i=0,&x\in\partial{\Omega}_i,\\
   p_i>0,&x\in\overline{\Omega}_{i},\\
   \int_{\Omega_i}p_i(x)dx=1.
      \end{cases}
  \end{equation*}

\end{rmk}

In the next remark, we comment on the assumptions made in Theorem \ref{thm:Global_attractors_BD}.
\begin{rmk}
\label{rem: assumptions}
In case a) the Lyapunov diagonal stability condition actually implies that there is a unique globally stable equilibrium with positive components. For b) the existence of such a state is an assumption, while for c) we do not need to assume the existence of such a state (in fact, we show the there is globally stable equilibrium, where all but one component is $0$).
\end{rmk}
We can interpret Theorem \ref{thm:Global_attractors_BD} biologically as follows: in cases a) and b) the total population mass $\rho$ behaves asymptotically the same as for the generalised Lotka-Volterra ODE system, with the same interaction matrix $A$, but where the intrinsic birth rate of species $i$ is determined by the principal eigenvalue $\lambda_i$. In case c) where the effect of species $j$ on $i$ depends only on species $i$, coexistence is typically not possible. Instead, the subpopulation $\underline{i}$, which corresponds to the minimal of \{$\lambda_i/a_i\}$ over all species $i\in\{1,...,N\}$, survives while all others go extinct. One can then interpret $\frac{\lambda_i}{a_i}$ as the fitness of the $i$-th species, and the model therefore exhibits survival of the fittest (Darwin's theory of natural selection).  The dependence of this fitness on trait-dependent birth rate $r_i$, the domain topology, and the diffusion coefficient is made clearer by considering the Rayleigh quotient representation for the principal eigenvalue,
\begin{equation}\label{eqn:Rayleigh}
    \lambda_i(r_i,d_i,\Omega_i)=\inf_{\phi\in{}H^1(\Omega_i)\setminus\{0\}}\frac{d_i\int_{\Omega_i}|\nabla\phi(x)|^2dx-\int_{\Omega_i}r_i(x)\phi(x)^2dx}{\int_{\Omega_i}\phi(x)^2dx}.
\end{equation}

From this formula, we deduce the following statements.
\begin{enumerate}
    \item  Since $\lambda_i(d_i,\Omega_i)$ is increasing in $d_i$, increased mutation rate is detrimental to survival. If we additionally suppose that $a_i=a$, $\Omega_i=\Omega$ and $r_i=r(x)$ for each $i\in\{1,...,N\}$ then this can be interpreted as a competition between $N$ subpopulations which differ only in their mutation rate, and it follows that the subpopulation with smallest mutation rate survives. {This is because the population with a smallest mutation rate is more strongly concentrated around the optimal trait, and therefore has a greater relative fitness. However, a non-zero mutation rate is necessary for that population to generate mass on optimal trait in the first place (e.g., without mutation, it would be impossible to produce offspring with the optimal trait if the initial support did not contain the optimal trait).}

    \item Since for $\Omega\subset\Omega^*$ we have that $\lambda_i(d,\Omega)<\lambda_i(d,\Omega^{*})$, increasing the explored trait-space is always beneficial (or neutral) in the sense that it allows the species to find potentially better traits.
\end{enumerate}
The first point is analogous to the results from the evolution of dispersal  where only the slowest diffusing subpopulation survives in a temporally constant environment \cite{dockery1998evolution,cantrell2021evolution}. In our case, with all else being equal, only the subpopulation with minimal mutation rate survives. We expect that this may differ in the case of periodic environment in line with \cite{liu2022new}, and show that it is different for the second submodel which considers linearly shifting optimal traits.

A corollary of Theorem \ref{thm:Global_attractors_BD} provides the long term behaviour for the $N=2$ case when either species is not viable alone ($\bar{r}_{i,\infty}<0$ for $i=1,2$) or when there is a stable co-existence equilibrium. 

\begin{corollary}
When $N=2$, the long time dynamics of $(\rho_1,\rho_2)$ are as follows.
If there is $i\in\{1,2\}$ such that $\bar{r}_{i,\infty}<0$ then:
\begin{itemize}

    \item  $\rho_i(t)\xrightarrow[t\rightarrow\infty]{}0$. Consequently, for $i'\in\{1,2\}\setminus\{i\}$, if $\bar{r}_{i,\infty}>0$, then $\rho_{i'}(t)\xrightarrow[t\rightarrow\infty]{}\frac{\bar{r}_{i'}}{a_{i'i'}}$. 
    \end{itemize}
Suppose instead that $\bar{r}_{i,\infty}>0$ for $i=1,2$ then:
\begin{itemize}
    \item  If $\frac{a_{21}}{a_{11}}<\frac{\bar{r}_2}{\bar{r_1}}<\frac{a_{22}}{a_{12}}$
    then the solution converges to the unique coexistence state $(\rho_1,\rho_2)\xrightarrow[t\rightarrow\infty]{}\left(\frac{\bar{r}_1a_{22}-\bar{r}_2a_{12}}{a_{11}a_{22}-a_{21}a_{12}},\frac{\bar{r}_2a_{11}-\bar{r}_1a_{21}}{a_{22}a_{11}-a_{21}a_{12}}\right)$.

\end{itemize}

\begin{rmk}
    In the case that there is a globally stable equilibrium on the boundary proof works but with an alternative Lyapunov function. However, the global dynamics in the case that   $\frac{a_{22}}{a_{12}}<\frac{\bar{r}_2}{\bar{r_1}}<\frac{a_{21}}{a_{11}}$ where $(p_1,0)$ and $(0,p_2)$ are locally stable equilibria for \eqref{eqn:MutationEvolution} is more intricate and worth investigating for future works. A natural question is how the trait-structure and mutation rates modify the the separatrix between these states. 
    
\end{rmk}
   
\end{corollary}

\subsection{Interacting species with shifting optimal traits}
Now we consider the second class of models which consist of $N$-species with shifting optimal traits. The system \eqref{eqn:General} becomes
\begin{equation}\label{eqn:ShiftingEnvironment}
      \begin{cases}
          \partial_{t}u_i=d_i\Delta{}u_i+u_i\left(r_i(x_i-\tilde{c}_it)-\sum_{j=1}^{N}a_{ij}\rho_j(t)\right),&(x_i,t)\in\mathbb{R}^n\times\mathbb{R}^{+},
    \\
\rho_i(t)=\int_{\mathbb{R}^n}u_i(x_i,t),&t\in\mathbb{R}^{+},\\
    
    u_i(x_i,0)=u_{i,0}(x_i)&x_i\in\mathbb{R}^n,
      \end{cases}
  \end{equation}
where $\tilde{c}_i\in\mathbb{R}^n$, $r_i\in{}C^2(\mathbb{R}^n)$. In additional to the usual assumptions that there are constants $C_1,C_2>0$ such that $0\leq{}u_{i,0}(x_i)\leq{}e^{C_1-C_2|x|}$ for $i\in\{1,2,...,N\}$, and that  $u_{i,0}(x_i)\in{}C^{0}(\mathbb{R}^n)$ is not identically $0$. Also, we assume that there exist constants $R,d>0$ such that  $r_i(x_i)<-d$ for $|x_i|>R$. Survival of a species now depends on both competition and how well it is adapting to the changing environment. As before, our aim is to determine the global dynamics for specific choices of the interaction matrix.

To state our first result for this model, we introduce the principle eigenvalue problems 
(where $i\in\{1,...,N\}$)
\begin{equation}\label{eqn:eigenfunctions}
      \begin{cases}
        -d_i\Delta{}p_i-\tilde{c}_i\cdot{}\nabla{}p_i=p_i\left(r_i(x_i)+\lambda_i\right),&x_i\in\mathbb{R}^n,
    \\
   p_i>0,\\
   \int_\mathbb{R}p_i=1,
      \end{cases}
  \end{equation}
The existence of solutions $(p_i,\lambda_i)$ has been proved in \cite{adaptivedynamicsdivergingfitness} via a construction from a sequence of problems in bounded domains. We use the same notation for the eigenvalue problem in each section since these are self-contained sections and it is clear when we are dealing with solutions to \eqref{eqn:eigenfunctionsDomain} or solutions to \eqref{eqn:eigenfunctions}. 

 Similar to the previous section, we define the following auxiliary systems. First, by using the transformation
\begin{equation}\label{eqn:DecouplingEqn_UBD}
    \tilde{u}_i(x,t)=u_i(x+\tilde{c}_it,t)e^{\int_{0}^ t\sum_{j=1}^Na_{ij}\rho_j(s)ds+\lambda_it}.
\end{equation}
we obtain the following equations for $\tilde{u}_i$:
\begin{equation}\label{eqn:Decoupled_UBD}
\begin{cases}
\partial_{t}\tilde{u}_i=d_i\Delta{}{\tilde{u}}_i+\tilde{c}_i\cdot\nabla\tilde{u}_i+\tilde{u}_i\left(r_i(x_i)+\lambda_i\right),&(x,t)\in\mathbb{R}^n\times\mathbb{R}^{+},
    \\
    u_i(x_i,0)=u_{i,0}(x_i),&x_i\in\mathbb{R}^n.
      \end{cases}
  \end{equation}
Next, by integrating \eqref{eqn:Decoupled_UBD} with respect to $x_i$, we obtain
\begin{align}\label{eqn:MassRelation_UBD}
\tilde{\rho}_i(t)=\int_{\mathbb{R}} \tilde{u}_i(x_i,t)\,dx=\rho_i(t)e^{\int_{0}^t\sum_{j=1}^Na_{ij}\rho_j(s)ds+\lambda_it},~t\in[0,\infty),
\end{align}
and finally, the corresponding generalised Lotka-Volterra ODE system
\begin{equation}\label{eqn:GLV_UBD}
    \frac{d}{dt}\hat{\rho}_i^ {}=\hat{\rho}_{i}^{}\left(-\lambda_i-\sum_{j=1}^{N}a_{ij}\hat{\rho}_j^{}\right),~~\text{for }~i\in\{1,...,N\}.
\end{equation}
On the existence of solutions to \eqref{eqn:MassRelation_UBD}, we have the following analogous result to Lemma \ref{lma:MassEquivalent}. 
\begin{lemma}\label{lma:MassEquivalent_UBD}
Assume that any solution $\hat{\rho}$ to 
\eqref{eqn:GLV_UBD}
with non-negative initial condition exists for all time $t\in\mathbb{R}_{>0}$.

{Also, assume that}

\begin{itemize}
   \item[(G)] {The solution to}
        \begin{align*}
        \frac{d}{dt}\rho_i(t)=\rho_i(t)\left(-\lambda_i-\sum_{j=1}^Na_{ij}\rho_j(t)+\alpha_i(t)\right),
    \end{align*} 
     {is bounded for all $t\geq{0}$, where $\alpha_i(t):=-\frac{d\tilde{\rho}_i(t)}{dt}\frac{1}{\tilde{\rho}_i(t)}$, $\tilde{u}=(\tilde{u}_1,...,\tilde{u}_N)$ is the solution to \eqref{eqn:Decoupled_UBD}, and $\tilde{\rho}_i=\int_{\mathbb{R}}\tilde{u}_i(x_i,t)dx$.}
\end{itemize}

Then there is a unique $\rho=(\rho_1,...,\rho_N)$ which satisfies \eqref{eqn:MassRelation_UBD} with $\rho_i\in{}C^1(\mathbb{R}_{>0})$. Consequently, there is a unique solution $u=(u_1,...,u_N)$ to \eqref{eqn:ShiftingEnvironment} where  $u_i\in{}C^{2,1}(\Omega\times\mathbb{R}_{>0})\cap{}C^{1,0}(\bar{\Omega}\times[0,\infty))$.
\end{lemma}

The second main result of the present paper is the following result, which characterises the long-time behaviour of the system \eqref{eqn:ShiftingEnvironment}. 
\begin{theorem}\label{thm:Global_attractors_UBD}
Let $u=(u_1,...,u_N)$ be the solution to \eqref{eqn:ShiftingEnvironment} and $\rho_i(x)=\int_{\mathbb{R}^n}u_i(x,t)dx$ for $i\in\{1,...,N\}$.

If $\rho^{*}$ is the unique equilibrium solution to \eqref{eqn:GLV_1} with positive entries, and one of
\begin{itemize}
    \item[a)]  $A$ is Lyapunov diagonally stable,
    \item[b)]  $a_{ij}>0$ for all $i,j\in\{1,...,N\}$, or $a_{ij}\leq{0}$ for $i\neq{j}\in\{1,...,N\}$
\end{itemize}
then $\rho(t)\xrightarrow[t\rightarrow\infty]{}\rho^{*}$ and $\Vert{}\frac{u_i(\cdot,t)}{\rho_i(t)}-p_i\Vert_{L^\infty(\mathbb{R}^n)}\xrightarrow[t\rightarrow\infty]{}0$.

In case c), where $a_{ij}=a_i>0$,  for each $i,j\in\{1,...,N\}$, we let $\underline{i}\in\{1,...,N\}$ be the unique index such that $\frac{\lambda_{\underline{i}}}{a_{\underline{i}}}=\min_{i\in\{1,...,N\}}\frac{\lambda_i}{a_i}<0$. Then
   \begin{align*}
       &\left\Vert{u_j}\right\Vert_{L^\infty(\mathbb{R}^n)}\xrightarrow[t\rightarrow\infty]{}0,~~\text{for }i\neq{}\underline{i}\\
       &\rho_{\underline{i}}(t)\xrightarrow[t\rightarrow\infty]{}-\frac{\lambda_{\underline{i}}}{a_i},~~ \left\Vert{}\frac{u_{\underline{i}}(\cdot,t)}{\rho_{\underline{i}}(t)}-p_i\right\Vert_{L^\infty(\mathbb{R}^n)}\xrightarrow[t\rightarrow\infty]{}0.
   \end{align*}

\end{theorem}

The biological interpretation is largely the same as the first model (see the discussion below Remark \ref{rem: assumptions}) except in the case c) where the eigenvalue might no longer be a monotonically increasing function of $d_i$. {For simplicity, we now consider the one dimensional case ($n=1$) where $\tilde{c}_i$ is a positive constant.} It is easily established \cite{iglesias2021selection,adaptivedynamicsdivergingfitness} that
\begin{equation}\label{eqn:RayleighUBD}
    \lambda_i(d_i,\tilde{c}_i)=\inf_{\phi\in{}H_0^1(\mathbb{R})/\{0\}}\frac{d_i\int_{\mathbb{R}}|\partial_{x}\phi(x)|^2dx-\int_{\mathbb{R}}r_i(x)\phi(x)^2dx}{\int_{\mathbb{R}}\phi(x)^2dx}+\frac{\tilde{c}_i^2}{4d_i},
\end{equation}
which can be written as
\[\lambda_i(d_i,\tilde{c}_i)=\lambda_i(d_i,0)+\frac{\tilde{c}_i^2}{4d_i}.\]
If $c_i\neq{0}$ then $\tilde{c}_i$, $\lambda_{i}(d_i,\tilde{c}_i)$ is minimal for an intermediate value of $d_i$. This means that in a changing environment, an intermediate mutation rate could be selected for provided that the environment changes sufficiently fast.

\subsection{Comparison to results for related models}
It is now instructive to compare the results of this paper with existing results for related models: a trait-structured generalised Lotka-Volterra system (without mutation) \cite{pouchol2018global}; models for the evolution of passive dispersal in static and dynamic environments \cite{dockery1998evolution}; and the single species version of \eqref{eqn:ShiftingEnvironment} but where the fitness function has multiple locally optimal traits that shift in different directions \cite{adaptivedynamicsdivergingfitness}.

In \cite{pouchol2018global}, the authors consider systems of the form
\begin{align}\label{eqn:NoMutation}
  \begin{cases}
      \partial_tn_i=n_i\left(r_i(x)+m_i(x)\sum_{j=1}^Na_{ij}\rho_j\right),&(x,t)\in\Omega_i\times{}\mathbb{R}_{>0},\\
    \rho_j(t)=\int_{\Omega_j}n_j(x,t)dx,&t\in\mathbb{R}_{>0}\\
  \end{cases}  
\end{align}
for $i=1,...,N,$ where $\Omega_i$ is a compact subset of $\mathbb{R}^n$. This model is identical to model \eqref{eqn:MutationEvolution} except model \eqref{eqn:MutationEvolution} includes simple diffusion (representing mutation between traits) and sets $m_i\equiv{1}$ so that interactions are uniform between traits. Both of these differences substantially alter the analytical approach required for the two problems: firstly, due to the lack of diffusion, the long-time behaviour of \eqref{eqn:NoMutation} depends strongly on the initial condition since the support of the solution is invariant. Secondly, the $x$-dependence in the interaction terms means there is no transform like \eqref{eqn:DecouplingEqn} that produces a decoupled system. Consequently, the approach based on generalised principal eigenvalue-eigenfunction pairs is no longer appropriate. Nevertheless, many conclusions hold for both models:
by making use of a Lyapunov functional (and assuming Lyapunov diagonal stability, e.g., case (a) in the present work), it is shown (\cite[Theorem 2.1]{pouchol2018global}) that the population  mass vector $\rho$ obtained from \eqref{eqn:NoMutation} converges to the equilibrium solution to a generalised Lotka-Volterra equation. Specifically, it holds that $\rho\xrightarrow[t\rightarrow\infty]{}\rho^\infty\in\mathbb{R}^N$ where $\rho^\infty$ solves $A\rho{}^\infty+I^\infty=0$ (where the $i$th component of $i$ is given by $I^\infty_i=\max_{i\in\Omega_i}\frac{r_i(x)}{m_i(x)}$).  Note that the principal eigenvalues for \eqref{eqn:eigenfunctions} satisfy $\lambda_i(d_i)\xrightarrow[d_i\rightarrow{0+}]{}-\max_{i\in\Omega_i}r_i(x)$.

Next, we remark that the conclusion of the previous section can be considered analogous to results obtained in studying the evolution of passive dispersal which was initiated in \cite{hastings1983can}. This was extended to the Lotka-Volterra framework in \cite{dockery1998evolution} where the following system was studied
\begin{align}\label{eqn:DispersalEvolution}
  \begin{cases}
      \partial_tn_i=d_i\Delta{}n_i+n_i\left(r(x)-\sum_{j=1}^N n_j\right),&(x,t)\in\Omega\times{}\mathbb{R}_{>0},\\
      \partial_{\nu}n_i=0,&x\in\partial\Omega.
  \end{cases}  
\end{align}
The above equation models  $N$ species which are identical except for their (now spatial) passive diffusion rates, competing in a spatial domain  $\Omega$ whose heterogenous resource distribution is represented by $r=r(x)$. Here the diffusion coefficients are ordered $d_1<d_2<...<d_N$. This has an identical form to \eqref{eqn:MutationEvolution} with uniform inter-population competitive interactions, except: the competition is now entirely local in $x$, and each population has the same reproduction rate $r$. The competitive exclusion by the slowest diffusers (e.g., the survival of only the slowest diffusing population) was shown in \cite{dockery1998evolution} for $N=2$. The question of whether the slowest diffuser competitively excludes all other populations for $N\geq{}3$ remains open, but see \cite{cantrell2021evolution} for progress towards a resolution. It should be noted that Theorem \ref{thm:Global_attractors_BD} (in case (c), where we set $r_i\equiv{}r$, $\Omega_i\equiv\Omega$, $a_i\equiv1$, for all $i=1,...,N$), together with the variational formula for the eigenvalue \eqref{eqn:Rayleigh}, implies that the weakest mutator (i.e., the population with smallest $d_i$) competitively excludes the other populations for any $N$.

The situation is made more complicated when the environment is inhomogeneous in time as well as space, e.g., where $r\equiv{}r(x,t)$ in \eqref{eqn:DispersalEvolution}. For two identical populations except for their diffusion coefficients and for $r(x,t)$ periodic in $t$, it was shown  that there are choices of $r$ such that any of the following outcomes is possible: both populations coexist; the slowest diffuser competitively excludes the fastest; the fastest diffuser competitively excludes the slowest \cite{hutson2001evolution}. The conditions under which each outcome occurs (given that either diffusion coefficient is sufficiently large or small) are presented in \cite{bai2022dynamics}. For model \eqref{eqn:ShiftingEnvironment}, part (c) of Theorem \ref{thm:Global_attractors_UBD}, together with the variational formula for the eigenvalue \eqref{eqn:RayleighUBD}, imply that an intermediate diffusion rate is always optimal.

To conclude, we discuss the possibility of analysing systems of the form \eqref{eqn:ShiftingEnvironment} when the reproduction rate $r$ depends on time and trait in more complex ways. In \cite{adaptivedynamicsdivergingfitness} the following equation was considered
\begin{equation}
    \begin{cases}
        \partial_tn-d\,\partial_{xx}n=n(r-\rho),&(x,t)\in\mathbb{R}\times\mathbb{R}_{>0},\\
        \rho(t)=\int_{\mathbb{R}}n(x,t)dx,&t\in\mathbb{R}_{>0},
    \end{cases}
\end{equation}
where $r=r_1(x-c_1t)+r_2(x-c_2t)-d$ and $r_1,r_2\in{}C^2_c(-R,R)$ for some $R>0$. Although this equation is intrinsically non-autonomous (there is no coordinate transform that exchanges the time-dependence of the reaction for a drift term), sub- and super-solution methods may be used to show that the solution will decay on the support of either $r_1(x-c_1t)$ or $r_2(x-c_2t)$ and thus determine the long-time behaviour. There are no major difficulties in considering reproduction rates of this form for each species in \eqref{eqn:ShiftingEnvironment} using the same approach since  comparison theorems are available for the linearised and decoupled equations.

\subsection{Numerics}

In this section, we provide numerics to illustrate Theorem \ref{thm:Global_attractors_BD} and Theorem \ref{thm:Global_attractors_UBD}. For all results in this section we let: $N=3$, $d_1=6/5$, $d_2=1$, and $d_3=4/5$; $r_1(x)=2\left(1-\left(x+1/10\right)^2\right)$, $r_2(x)=2\left(1-(3/2)|x|^3\right)$, and $r_3(x)=3\left(1-\frac{3}{2}\left(x-(1/10)\right)^2\right)$ for $x\in\Omega$.  We take $A=A_j$ for $j\in\{1,2,3,4\}$ where
\begin{enumerate}
    \item[(a)] $A_1=\begin{pmatrix}
        1 & 0 & 1\\
        1 &2&-\frac{3}{2}\\
        -2 & \frac{1}{2}&3
    \end{pmatrix}$ for the Lyapunov diagonally stable case,
    \item[(b)] $A_2=\begin{pmatrix}
        1 & 0 & -1\\
        0 &2&-1\\
        -\frac{1}{10} & -\frac{1}{2}&3
    \end{pmatrix}$ for the pure cooperation case, $A_3=\begin{pmatrix}
        1 & 0 & 1\\
        0 &2&1\\
        \frac{1}{10} & \frac{1}{2}&3
    \end{pmatrix}$
    for the pure competition case, or
    \item[(c)] $A_4=\begin{pmatrix}
        \frac{3}{5} & \frac{3}{5} & \frac{3}{5}\\
        1 &1&1\\
        \frac{4}{5} & \frac{4}{5} & \frac{4}{5}
    \end{pmatrix}$ for the uniform competition case.
\end{enumerate}
We consider the problem in a bounded domain $\Omega=[-1,1]$ for numerical results related to Theorem \ref{thm:Global_attractors_BD}) and the unbounded domain $\Omega=\mathbb{R}$ (for numerical results related to Theorem \ref{thm:Global_attractors_UBD}).

For both domains $\Omega=[-1,1]$ and $\Omega=\mathbb{R}$, we solve {\eqref{eqn:MutationEvolution} and \eqref{eqn:ShiftingEnvironment} on an interval of the form $[-L,L]$. We discretize the $x$-domain uniformly with spacing $dx=\frac{1}{200}$, and discretize the time-domain $[0,T]$ with uniform spacing $dt=\frac{dx^2}{20}$ (where we specify $T$ in each section below). For $i\in{}\{1,...,N_t/dt\}$ and $j\in\{0,...,2L/dx\}$ $x_j=-L+jdx{}$ and $t_k=kdt$ (where $i=1,...,N_x$ and $k=1,...,N_t$)  the solution $u^{k}_j=(u_{1,j}^k,u_{2,j}^k,u_{3,j}^k)$ at point $(x_j,t_k)$  is obtained using a simple finite difference scheme analogous to the one employed in \cite{adaptivedynamicsdivergingfitness}.  The coupling between components occurs only through the non-local terms $\rho^k=(\rho_1^k,\rho_2^k,\rho_3^k)$ which are obtained by integration at each time-step. As initial conditions we take $u(x)=\frac{3}{5}e^{-x^2}(1,1,1)$ for $x\in\Omega$.

In order to compute the eigenfunction-eigenvalue pairs $(p_i,\lambda_i)$ (solving either \eqref{eqn:eigenfunctionsDomain} or \eqref{eqn:eigenfunctions}) we  employ the Strum-Lioville toolbox Matslise 2.0 \cite{ledoux2016matslise}. Using the numerically obtained eigenvalues, we then compute the solution  $\tilde{\rho}=(\tilde{\rho}_1,\tilde{\rho}_2,\tilde{\rho}_3)$ to the Generalised Lotka-Volterra ODE \eqref{eqn:GLV_1}  using \texttt{MATLAB}'s ode45 solver.

We first present the results related to the problem on the bounded domain with Neumann boundary conditions \eqref{eqn:MutationEvolution} whose long time asymptotics are given in Theorem \ref{thm:Global_attractors_BD}.  \begin{figure}[ht]
\centering
    \subfigure[]{\includegraphics[scale=.45]{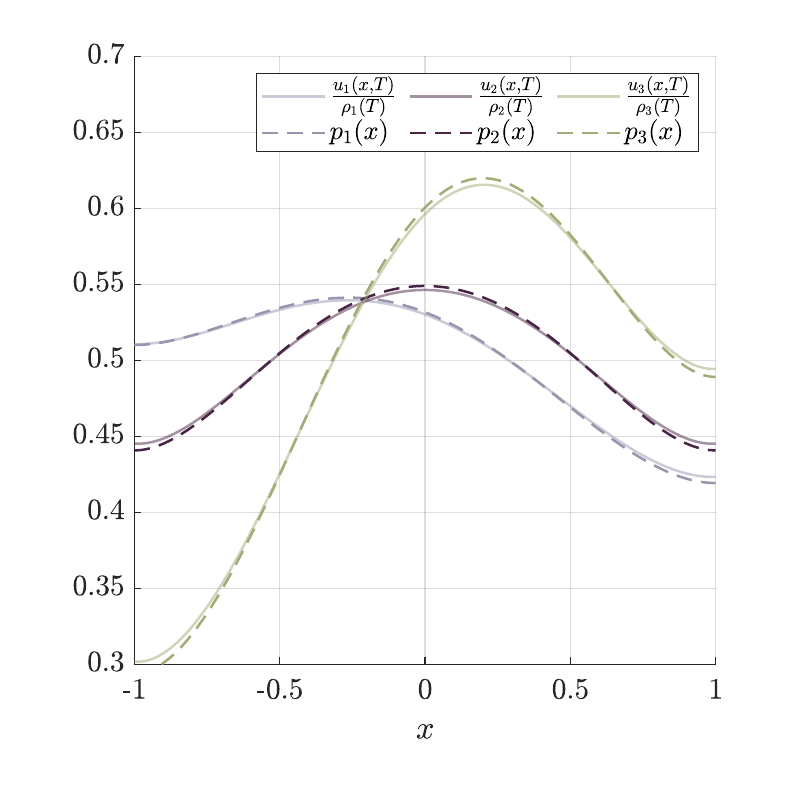}}
    \subfigure[]{\includegraphics[scale=.325]{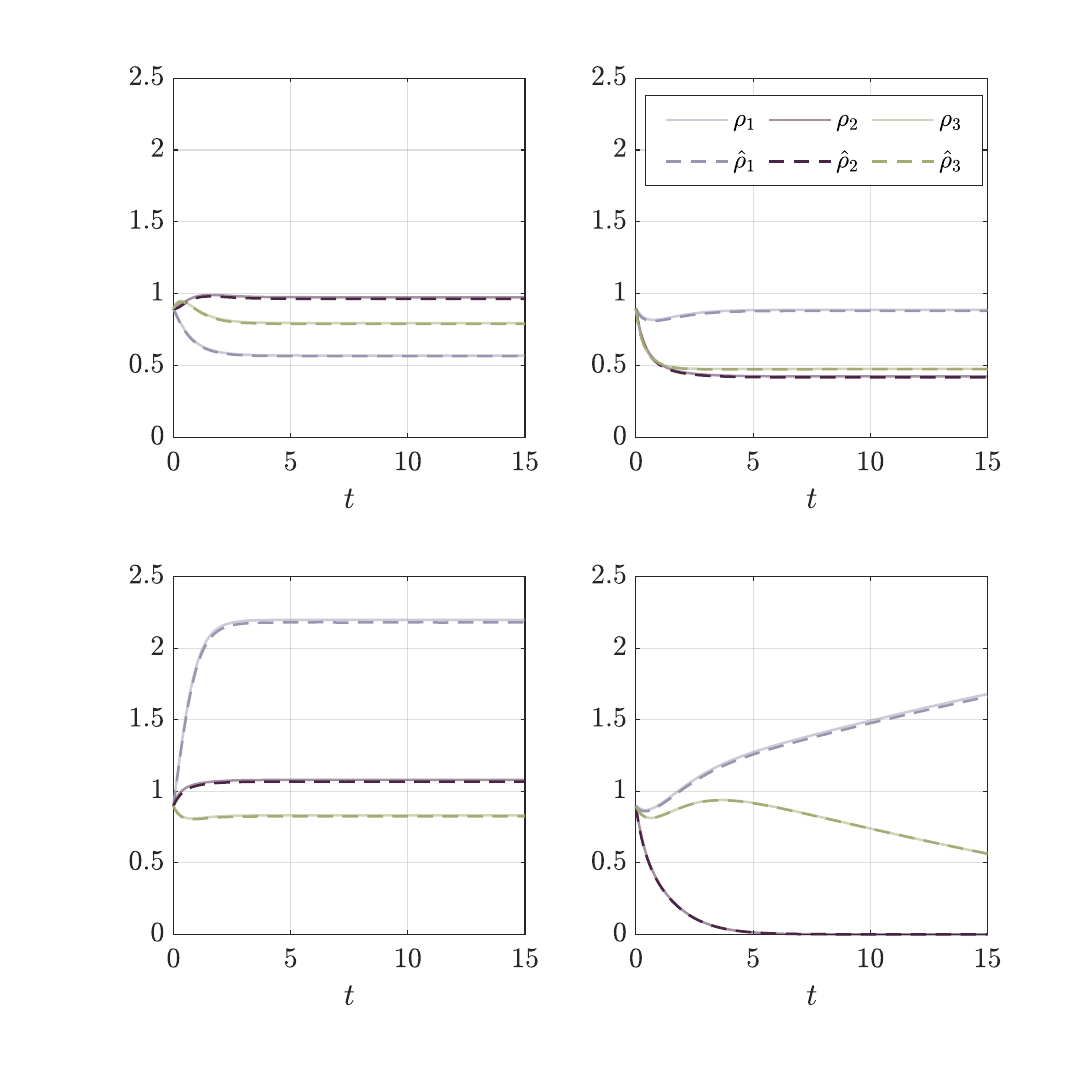}}
    \caption{Plots of trait distributions and the evolution of mass vectors obtained from solving \eqref{eqn:MutationEvolution}. \textbf{(a)}  The trait distributions $u_i(x,T)/p_i(T)$ (for $i=1,...,3$) (solid lines) and the eigenfunction associated to the linearised PDE (dotted) for $A=A_1$ obtained from. \textbf{(b)}  For $A=A_1,A_2,A_3$ or $A_4$, in order top left, top right, bottom left, bottom right: Evolution of population mass vector $\rho=(\rho_1,\rho_2,\rho_3)$ (solid lines) and the solution $\hat{\rho}=(\hat{\rho}_1,\hat{\rho}_2,\hat{\rho}_3)$  to \eqref{eqn:GLV_1} (dashed lines) up to $T=15$.}
\label{fig:Theorem21}
\end{figure}Figure (\ref{fig:Theorem21}a) shows the trait distributions $\frac{u_j(x,T)}{\rho_j(T)}$ (for $j=1,2,3$) at the final time $T=15$ for $A=A_1$ which are in close agreement with the eigenfunctions $p_j$ which solve \eqref{eqn:eigenfunctionsDomain} in accordance with Theorem \ref{thm:Global_attractors_BD}. We only consider the results for $A=A_1$ since the distributions are identical in the other cases. In Figure (\ref{fig:Theorem21}b) we compare the evolution of the solution to \eqref{eqn:GLV_1} and the population mass vector $\rho$ obtained from \eqref{eqn:MutationEvolution}. The long time behaviour of both mass vectors are the same, although for the chosen parameters and initial conditions they actually remain close for all times. For the uniform competition case (bottom right panel of Figure (\ref{fig:Theorem21}b)),  both $\rho_2$ and $\rho_3$ vanish as $t\rightarrow\infty$ which is accordance with the final statement in Theorem \ref{thm:Global_attractors_BD} since (using the numerical computed eigenvalues) we have $\frac{\lambda_1}{a_1}=-2.2590...<\frac{\lambda_3}{a_3}=-2.1518...<\frac{\lambda_2}{a_2}=-1.3100...$ so that $1=\text{argmin}_{i\in\{1,2,3\}}\frac{\lambda_i}{a_i}$.

To illustrate Theorem 2.2 numerically we approximate  \eqref{eqn:ShiftingEnvironment} on the infinite domain $\Omega=\mathbb{R}$ by the Dirichlet problem on the finite domain $[-10,10]$ (after shifting coordinates to the moving frame $x_i\xrightarrow[]{}x_i-\tilde{c}_i{}t$).  We set $\tilde{c}_1=1/10$, $\tilde{c}_2=1/20$, and $\tilde{c}_3=1/5$.  Figure (\ref{fig:Theorem22}a) demonstrates the agreement between the trait distributions $u_i(x,T)/p_i(T)$ obtained from numerical solution to the Dirichlet problem that approximates \eqref{eqn:ShiftingEnvironment} and the eigenfunctions $p_i$ that solve \eqref{eqn:eigenfunctions} (for $i=1,2,3$). Figure (\ref{fig:Theorem22}b) shows discrepancies at early times between the mass vector $\rho=(\rho_1,\rho_2,\rho_3)$ and the solution to the GLV \eqref{eqn:GLV_1} (these are most pronounced in the bottom two panels) but the long time asymptotics remain identical as per Theorem \ref{thm:Global_attractors_UBD}. Similarly, the population which minimises $\frac{\lambda_i}{a_i}$ over $i\in\{1,2,3\}$ survives in the uniform competition case (Figure \ref{fig:Theorem22}b, bottom right panel), corresponding to $A=A_4$.
\begin{figure}[h]
\centering
    \subfigure[]{\includegraphics[scale=.45]{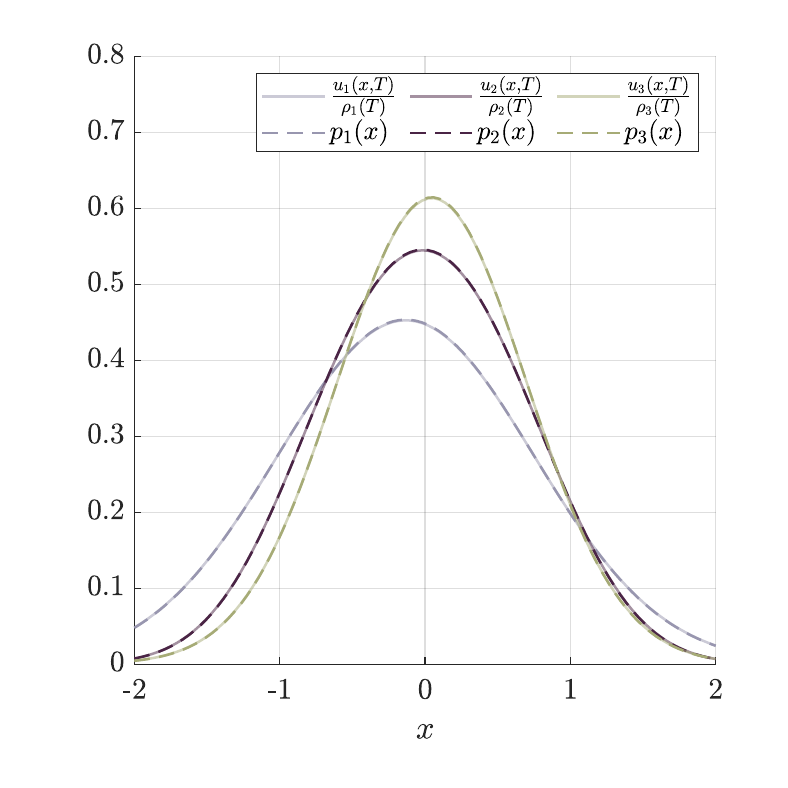}}
    \subfigure[]{\includegraphics[scale=.325]{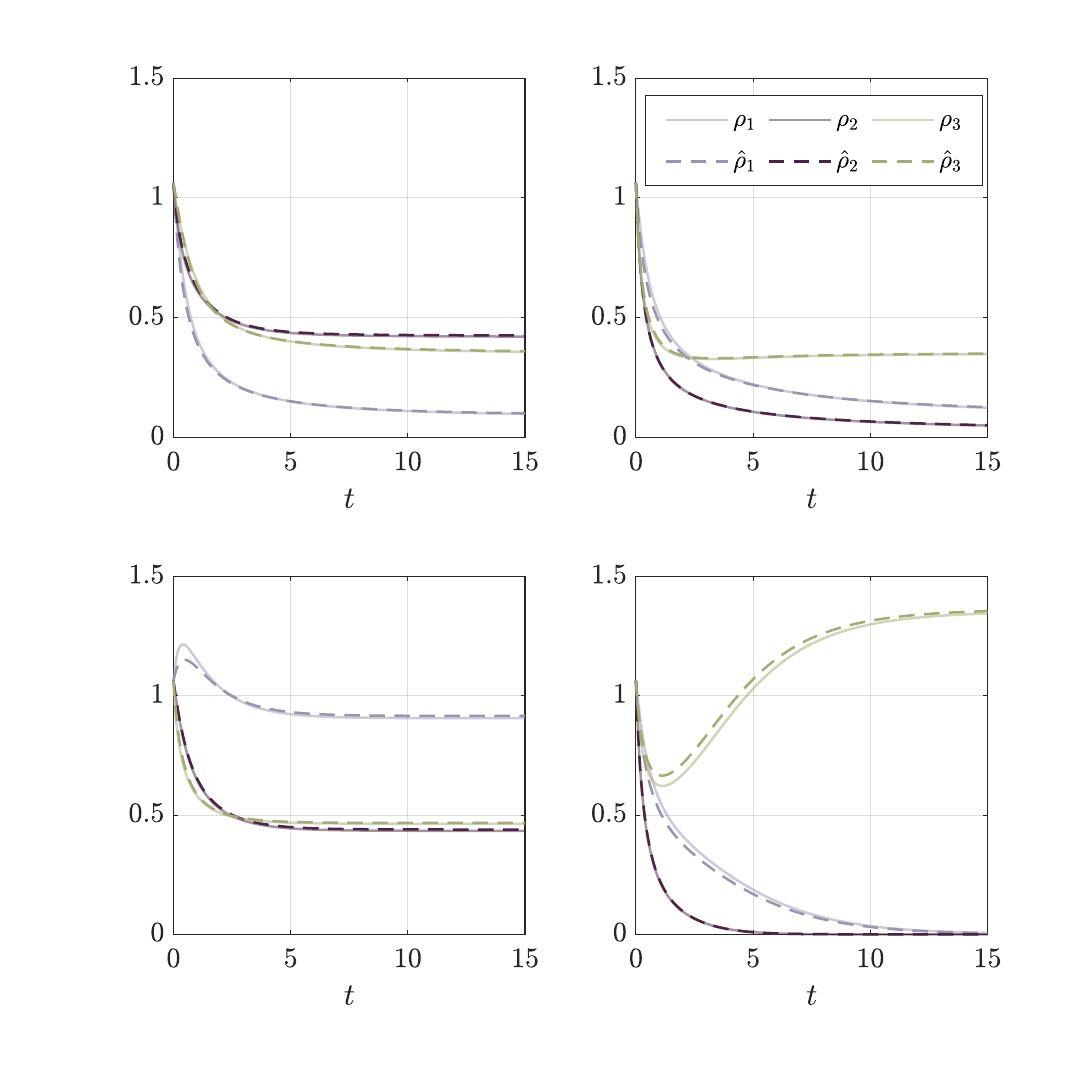}}
    \caption{Plots of trait distributions and the evolution of mass vectors obtained from solving \eqref{eqn:ShiftingEnvironment}. \textbf{(a)}  The trait distributions $u_i(x,T)/p_i(T)$ (for $i=1,...,3$) (solid lines) and the eigenfunction associated to the linearised PDE (dotted lines) for $A=A_1$ obtained from. \textbf{(b)}  For $A=A_1,A_2,A_3$ or $A_4$, in order top left, top right, bottom left, bottom right: Evolution of population mass vector $\rho=(\rho_1,\rho_2,\rho_3)$ (solid lines) and the solution $\hat{\rho}=(\hat{\rho}_1,\hat{\rho}_2,\hat{\rho}_3)$  to \eqref{eqn:GLV_1} (dashed lines) up to $T=15$.}
\label{fig:Theorem22}
\end{figure}

\section{Proofs of Theorems \ref{thm:Global_attractors_BD} and \ref{thm:Global_attractors_UBD}}\label{sec:Proofs}
In this section we provide the detailed proof of Theorem \ref{thm:Global_attractors_BD}, and since it follows very closely, outline the proof of Theorem \ref{thm:Global_attractors_UBD}. To prove Theorem \ref{thm:Global_attractors_BD}, we begin by establishing the long time behaviour of the solution $\tilde{u}$ to the linear, decoupled system \eqref{eqn:Decoupled} (Lemmas \ref{lma:TransformBounds}, \ref{lma:AlphaExpDecay}). The control on $\tilde{u}$ is then used to show the existence and uniqueness of a solution $u$ to \eqref{eqn:MutationEvolution} (Lemmas \ref{lma:MassEquivalent} and \ref{lma:BoundedTrajetories}}). Finally, we relate this to the long-time behaviour of $u$ (Lemma \ref{lma:MassBoundSolution}), and  use some Lyapunov functions, or the boundedness of trajectories, to prove Theorem \ref{thm:Global_attractors_BD}.

\subsection{Proof of Theorem \ref{thm:Global_attractors_BD}}
\label{sec: proof21}
We first state a lemma, which is a consequence of the comparison principle for parabolic equations (e.g., we use multiples of the principle eigenvalue as sub- and super-solutions).
\begin{lemma}\label{lma:TransformBounds}
    Given  $\tilde{u}(x,t)=(\tilde{u}_1(x,t),...\tilde{u}_2(x,t))$ solves \eqref{eqn:Decoupled}, and $t_0>0$, there exist constants $\underline{w}^ i\leq{}\overline{w}^i$ ($i=1,...,N)$ such that $\underline{w}_ip_i(x)\leq{}\tilde{u}_{i}(x,t)\leq{}\overline{w}_ip_i(x)$ for $(x,t)\in\Omega_i\times(t_0,\infty)$. If $u_{i,0}$ is not identically $0$, then $\underline{w}_i>0$.
\end{lemma}

\begin{proof}
Since for any constant $w>0$, both $u_i$ and $wp_i$ solve
    \begin{equation}
      \begin{cases}
          \partial_{t}\tilde{u}_i=d_i\Delta{\tilde{u}}_i+\tilde{u}_i\left(r_i(x_i)+\lambda_i\right),&(x_i,t)\in\Omega_i\times\mathbb{R}^{+},
    \\
    
    \partial_{\nu}\tilde{u}_i=0,&(x_i,t)\in\partial{\Omega_i}\times\mathbb{R}^{+},
      \end{cases}
  \end{equation}
  with respective initial conditions of $u_{i,0}$ or and $wp_i$ then, due to the the comparison theorem, if  $wp_i(x)\leq{}u_{i,0}(x)$ (resp. $wp_i(x)\geq{}u_{i,0}(x)$) for all $x\in\Omega_i$ then $wp_i(x)\leq{}u_i(x,t)$ (resp.  $wp_i(x)\geq{}u_i(x,t)$)  for all $(x,t)\in\Omega_i\times[0,\infty)$. Since $p_i>0$ on $\partial\Omega_i$ and is bounded below, we can always pick $\overline{w}^i>0$ such that $u_{i,0}\leq{}\overline{w}_ip_0$. 

  We now focus on the lower bound. If $u_{i,0}>0$ on $\partial\Omega_i$ then because $p_i$ is bounded above, we can pick $\underline{w}_i$ such that $\underline{w}_ip_i\leq{}u_{i,0}$ in which case we may choose $t_0=0$. Otherwise, given any $t_0>0$, the strong parabolic maximum principle (see, e.g., \cite[Theorem 1.1.9]{lam2022introduction}) and fact that $u_{i,0}$ is non-zero  immediately yields that $u(x,t_0)>0$ for any $(x,t)\in\partial\Omega$.

\end{proof}

The following lemma is key to controlling the error between the perturbed and unperturbed Generalised Lotka-Volterra systems. 

\begin{lemma}\label{lma:AlphaExpDecay}
    Where $\tilde{u}=(\tilde{u}_1,...,\tilde{u}_N)$ is the solution to \eqref{eqn:Decoupled} and $\tilde{\rho}_i=\int_{\Omega_i}\tilde{u}_i(x,t)dx$, then there exists $A,B,K>0$ such that $\Vert{}\tilde{u}_i(\cdot,t)-Kp_i\Vert_{L^\infty(\Omega_i)},\left|\frac{d\tilde{\rho}_i}{dt}\right|<e^{A-Bt}$.
\end{lemma}

\begin{proof}[Proof of Lemma \ref{lma:AlphaExpDecay}]
   From, for example \cite[Theorem 4.2.2]{lam2022introduction}, we have that there exists a
  $K\in\mathbb{R}$  such that \begin{equation}\label{eqn:ExponentialConvergenceFunction}
      \Vert{}\tilde{u}_i(\cdot,t)-Kp_i\Vert_{C(\overline{\Omega}_i)}\xrightarrow[t\rightarrow\infty]{}0
  \end{equation}and that this convergence is exponential (note this case is much simpler since the coefficients are time-independent). Moreover, $K>0$ due to Lemma \ref{lma:TransformBounds}. Integrating \eqref{eqn:Decoupled} with respect to $x$ then yields
  \[\frac{d\tilde{\rho}_i}{dt}=\int_{\Omega_i}\tilde{u}_i(x,t)r_i(x)dx+\lambda_i\tilde{\rho}_i,\]
whose implicit solution is $\tilde{\rho}_i(t)=\int_{0}^{t}e^{\lambda_i(t-s)}\int_{\Omega_i}\tilde{u}_i(x,s)r_i(x)dz.$ From \eqref{eqn:ExponentialConvergenceFunction}, we have that there are constants $A,B>0$ such that \begin{equation}
    \left|\tilde{\rho}_i(t)-K\right|<\frac{1}{2}e^{A-Bt}
\end{equation} 
Using \eqref{eqn:ExponentialConvergenceFunction} again we may write
  \[\frac{d\tilde{\rho}_i}{dt}=K\bar{r}_{i,\infty}+\lambda_i\tilde{\rho}_i+\beta(t),\]
  where $\bar{r}_{i,\infty}:=\int_{\Omega_i}r_i(x)p_i(x)=-\lambda_i$, where this last equality can be verified by integrating \eqref{eqn:eigenfunctionsDomain} $\beta(t)$ satisfies $\left|\beta(t)\right|<\frac{1}{2}e^{A-Bt}$ for potentially large $A>0$ and/or smaller (but still positive) $B$. The conclusion of the Lemma then follows immediately.



    
\end{proof}
Next we prove Lemma \ref{lma:MassEquivalent}.
\begin{proof}[Proof of Lemma \ref{lma:MassEquivalent}]

We construct a solution $\rho=(\rho_1,...,\rho_N)$ to \eqref{eqn:MassRelation} such that $\rho_i\in{}C^1(\mathbb{R}_{>0})$ for $i\in\{1,...,N\}$. Formally differentiating \eqref{eqn:MassRelation} we obtain
    \begin{align*}
        \frac{d\tilde{\rho}_i(t)}{dt}=\frac{d\rho_i(t)}{dt}\frac{\tilde{\rho}_i(t)}{\rho_i(t)}+\left(\sum_{j=1}^Na_{ij}\rho_j(t)+\lambda_i\right)\tilde{\rho}_j(t).
    \end{align*}
    Upon rearrangement we find that
        \begin{align*}
        \frac{d\rho_i(t)}{dt}=\rho_i(t)\left(-\lambda_i-\sum_{j=1}^Na_{ij}\rho_j(t)+\alpha_i(t)\right),
    \end{align*} 
    where $\alpha_i(t):=-\frac{d\tilde{\rho}_i(t)}{dt}\frac{1}{\tilde{\rho}_i(t)}$ converges to $0$ exponentially fast (following Lemma \ref{lma:AlphaExpDecay}, {from which it also follows that $\tilde{\rho}_i$ converges to the constant $K>0$})

    Note that
    this is an ODE of the form
    \[\frac{d\rho}{dt}=F(\rho,t),\]
where $F(\rho,t)$ continuous in $t$ (since $\tilde{\rho}_i\in{}C^1(\mathbb{R}_{>0})$) and $\frac{\partial{F}}{\partial{\rho_i}}$ is bounded in any open set $U\in{}\mathbb{R}^N_{>0}$.  The uniqueness and local existence then follows from the usual Picard-Lindelof Theorem without invoking the assumption (G).

Due to assumption (G), we have globally defined solution $\rho(t)$ to \eqref{eqn:MassRelation} and therefore also, via \eqref{eqn:DecouplingEqn}, a globally defined classical solution $u$ to \eqref{eqn:MutationEvolution}. This is necessarily unique since both $\rho$ and $\tilde{u}_i$ are uniquely determined by the initial condition.

\end{proof}

We are now in a position to prove Lemma \ref{lma:MassBoundSolution}.

\begin{proof}[Proof of Lemma \ref{lma:MassBoundSolution}]
    From  Lemma \ref{lma:TransformBounds} and \eqref{eqn:DecouplingEqn} we have
\begin{equation}\label{eqn:TransformBounds}
    \underline{w}_ip_i(x)\leq{}u_i(x,t)e^{\int_{0}^ t\sum_{j=1}^Na_{ij}\rho_j(s)ds+\lambda_it}\leq{}\overline{w}_ip_i(x),
\end{equation}
which we integrate with respect to $x$ to obtain
\begin{equation}\label{eqn:MassBounds}
    \underline{w}_i\leq{}\rho_i(t)e^{\int_{0}^ t\sum_{j=1}^Na_{ij}\rho_j(s)ds+\lambda_it}\leq{}\overline{w}_i.
\end{equation}
We now take logs to find
\[\log(\underline{w}_i)\leq{}\log(\rho_i(t))+t\left(\frac{1}{t}\int_{0}^ t\sum_{}a_{ij}\rho_j(s)ds+\lambda_i\right)\leq{}\log(\overline{w}_i).\]
Suppose that $\rho_i(t)$, and, consequently, $\log(\rho_i(t))$ are bounded above. Then there exists a constant  $\overline{c}_i$ such that
\[t\left(\frac{1}{t}\int_{0}^ t\sum_{}a_{ij}\rho_j(s)ds+\lambda_i\right)\leq{}\overline{c}_i,\]
from which the first part of the lemma follows. The second part follows analogously, and the third is a consequence of the preceding two parts.

\end{proof}

Before proving Theorem \ref{thm:Global_attractors_BD} we prove the following auxiliary Lemma which ensures that there exists a globally defined solution to \eqref{eqn:MutationEvolution} for each case we consider in that theorem.

\begin{lemma}\label{lma:BoundedTrajetories}
    The assumption (G) in Lemma \ref{lma:MassEquivalent} is satisfied in the three cases a)-c) considered in Theorem \ref{thm:Global_attractors_BD}.
\end{lemma}

\begin{proof}
    For parts b) (given $a_{ij}\geq{0}$ for $i\neq{j
}\in\{1,...,N\}$) and c): due to the positivity of all entries in the matrix $A$, we have that  $\frac{d\rho_i}{dt}<\rho_i(\bar{r}_{i,\infty}+\alpha_i(t)-a_{ii}\rho_i)$. Since $\alpha_{i}(t)$ is bounded (by Lemma \ref{lma:AlphaExpDecay}) we have $\frac{d\rho_i}{dt}<\rho_i(\bar{r}_{i,\infty}+\sup_{t\geq{0}}\alpha_i(t)-a_{ii}\rho_i)$. Consequently $Q=(0,\sup_{t\geq{0}}\alpha_1(t)+r_{1,\infty}]\times...\times(0,\sup_{t\geq{0}}\alpha_N(t)+r_{N,\infty}]$ is a trapping region and solutions are bounded. 

For part b), given $a_{ij}\leq{0}$ for each $i\neq{j}\in\{1,...,N\}$, under the assumption there is a unique solution $\rho^{*}$ to \eqref{eqn:GLV_ODE_generic}, is in fact a subcase of $a)$  (see \cite{pouchol2018global} Lemmas 4.1 - 4.3) which we will deal with next.


For part a),  we recall the usual Lyapunov function $V:\rho\in\mathbb{R}^n\mapsto{}V(\rho)\in\mathbb{R}$ for the generalised Lotka-Volterra system
 \[V(\rho)=\sum_{i=1}^NP_{ii}\left(\rho_i-\rho_i^{*}-\rho_{i}^{*}\log\frac{\rho_i}{\rho_i^{*{}}}\right),\]
 where $P=(P_{ij})$ is the positive diagonal matrix such that  such that $PA^T+AP$ is positive definite, and $\rho^{*}$ is the unique eqilibrium solution to \eqref{eqn:GLV_1}. It can be checked that each term in the sum is positive for all $t\geq{0}$ for $\rho_i>0$, implying $V(\rho)>0$ for all $\rho>0$ (component wise), and it is immediate that $V(\rho^{*})=0$.  

If  $V(\rho)<\infty$ for all $t\geq{0}$ then no component of $\rho$ can blow up in finite time.   Differentiating $V(\rho)$ with respect to $t$ yields
 \[\frac{dV(\rho)}{dt}=-\frac{1}{2}(\rho-\rho^{*})^T(PA^T+AP)(\rho-\rho^{*})+(\rho-\rho^{*})^TP\alpha.\]
Let \[C_1=\inf_{x\in\mathbb{R}^n,|x|=1}x^T(PA^T+AP)x>0,\text{ and }C_2=\sup_{x,y\in\mathbb{R}^n,|x|=|y|=1}x^TPy<\infty.\] Then we can estimate
 \[\frac{dV(\rho)}{dt}<-\frac{C_1}{2}|\rho-\rho^{*}|^2+C_2|\rho-\rho^{*}|\sup_{t\geq{0}}|\alpha|.\]
 The RHS is negative for $\rho$ large enough. Therefore, $V(\rho)<\infty$ for all $t\geq{0}$ and solutions are bounded.
\end{proof}

We now introduce the result from \cite{strauss1967asymptotically}  on asymptotically autonomous system, which when combined with condition (G), enables us to say the long-time behaviour of the non-autonomous systems is the same as for its their autonomous limit. To state it, we first consider the general differential equation  \begin{equation}\label{eqn:GenAut}
      \dot{y}=f(y),
  \end{equation}
   and 
     \begin{equation}\label{eqn:GenNonAut}
   \dot{y}=f(y)+g(y,t),
  \end{equation}
   where $f:\mathbb{R}^N\to\mathbb{R}^N$, and for $t\geq{0}$, $g(\cdot,t):\mathbb{R}^N\to\mathbb{R}^N$. 
   \begin{definition}
       Here $\vert\cdot\vert$ is any norm on $\mathbb{R}^N$. The function $g$ is said to mostly approach zero if for every compact set $B\subset\mathbb{R}^N$ there is a function $E_B:[1,\infty)\to\mathbb{R}^{+}$ such that $\lim_{s\rightarrow\infty}E_B(s)=0$ and
    \[\left\vert\int_{s}^{s+t}g(w(u),u)du\right\vert\leq{}E_B(s),\]
    for all $t\in[-1,1]$, $s\in[1,\infty)$, and $w:\mathbb{R}\to{B}$.
   \end{definition}

\begin{lemma}[Adapted from Theorem 2.4 and Corollary 3.3 of \cite{strauss1967asymptotically}]\label{lma:AsympAut}

    Suppose that $g$ mostly approaches $0$, $y^{*}\in\mathbb{R}^N$ is a global attractor for \eqref{eqn:GenAut}, and   $y(t)$ is a bounded solution to \eqref{eqn:GenNonAut}. Then
 $\lim_{t\rightarrow\infty}y(t)=y^{*}$.
\end{lemma}

All three parts a)-c) then follow from the next lemma. 

\begin{proof}[Proof of Theorem \ref{thm:Global_attractors_BD}:]
     We write the ODE system for the population masses 
\begin{equation}\label{eqn:PopDynFull_1}
    \begin{cases}
        \frac{d\rho_i}{dt}=\rho_i\left(-\lambda_i-\sum_{i=1}^{N}a_{ij}\rho_j\right)+\alpha_i(t)\rho_i,\\
        \rho_i(0)=\int_{\Omega_i}u_{i,0}(y)dy,
    \end{cases}
\end{equation}
where $\alpha_i(t)$ converges to $0$ exponentially fast for each $i\in\{1,...,N\}$. The conclusion follows from Lemma \ref{lma:AsympAut} once we show $g(t,\rho):=(\rho_1\alpha_1(t),...,\rho_N\alpha_N(t))$ mostly approaches $0$. We already know, by Lemma \ref{lma:BoundedTrajetories} that the solution $\rho$ to \eqref{eqn:PopDynFull_1} is bounded.

We take $|\cdot{}|_1$ to be the $l_1$ norm. Let $B$ be any compact set in $\mathbb{R}^N$ and $w$ any continuous function from $\mathbb{R}$ to $B$. Then
\begin{align*}
   \left\vert \int_{s}^{s+t}g(u,w(u))du\right\vert_1&= \sum_{i=1}^N\left\vert \int_{s}^{s+t}\alpha_i(u)w_i(u)du\right\vert\\
   &<N\sup_{x\in{B}}|x|_1\sum_{i=1}^N\int_{s}^{s+t}|\alpha_i(u)|du\\
   &<N^2\sup_{x\in{B}}|x|_1e^{A-B{s}},
\end{align*}
where $A,B>0$ are chosen as in Lemma \ref{lma:AlphaExpDecay} such that $\alpha_i(t)<e^{A-Bt}$ for $i=1,...,N$. This shows that $g$ mostly approaches $0$ and the conclusion follows.

\end{proof}

Alternatively, we could use the Lyapnov functions introduced to show the solution is bounded in cases a) and c). But we emphasise that there was no need to construct a Lyapunov function in case b).

\begin{proof}[Alternative proof of Theorem \ref{thm:Global_attractors_BD} part a)]


Suppose that $\rho(t)$ does not converge to $\rho^{*}$. If there is an open set $U\ni\rho^{*}$ such that for all large enough $t$, $\rho(t)\notin{}U$  then due to the positive-definiteness of $PA+AP^T$ and exponential decay of $\alpha$, we have that eventually $\frac{dV(\rho)}{dt}<-\frac{1}{3}(\rho-\rho^{*})^T(PA+AP^T)(\rho-\rho^{*})<0$.  But this would imply that $V(\rho(t))<0$ eventually which contradicts $V\geq{0}$, so no such open set exists. However, due to the continuity of the vector field, any such open set is eventually a trapping region.  Hence we must have $\rho(t)\xrightarrow[t\rightarrow\infty]{}\rho_{}^{*}$.

\end{proof}

\begin{proof}[Alternative proof of Theorem \ref{thm:Global_attractors_BD} part c)]
    We assume that $a_{ij}=a_i>0$ for $i,j\in\{1,...,N\}$. It is clear that $\rho_i(t)$ is bounded above since 
    \[\frac{d\rho_i}{dt}\leq{}\rho_i(r_{i,M}-a_i\rho_i),\]
    and hence there is a barrier at $\rho_i=\frac{r_{i,M}}{a_i}$. Therefore, from part 1 of Lemma \ref{lma:MassBoundSolution}, we have that either $\limsup_{t>0}a_i\frac{1}{t}\int_{0}^ t\rho_i(s)ds=-\lambda_i$ or $\rho_i(t)\xrightarrow[t\rightarrow\infty]{}0$.  We have assumed there is a unique (negative) minimum among $\frac{\lambda_1}{a_1}$, $\frac{\lambda_2}{a_2}$,...,$\frac{\lambda_N}{a_N}$, so without loss of generality let these be ordered $\frac{\lambda_1}{a_1}<\frac{\lambda_2}{a_2}\leq{}...\leq{}\frac{\lambda_N}{a_N}$ so that $\underline{i}=1$. Therefore  $\rho_i(t)\xrightarrow[t\rightarrow\infty]{}0$ for all $i\in\{2,...,N\}$. 
    
    We now use this fact to obtain the limit of $\rho_1(t)$. Let $I(t)=\frac{1}{\rho_{1}(t)}$ so that $I$ solves
    \[\frac{dI}{dt}=(\lambda_{1}-\beta(t)+\alpha_1(t))I+a_{1},\]
    where $a_{i}\sum_{j=1,j\neq{i}}^N\rho_j(t)=:\beta(t)\xrightarrow[t\rightarrow\infty]{}0$ and this can be shown to exponentially fast using \eqref{eqn:MassBounds} and the strict inequalities $\frac{\lambda_1}{a_1}<\frac{\lambda_j}{a_j}$ for $j\in\{1,...,N\}$.  It's clear that $I(t)=I(0)e^{\lambda_1t+\int_{0}^t\alpha_1(s)-\beta(s)ds}+a_{1}\int_{0}^t{}e^{\lambda_{1}(t-s)+\int_{s}^ t\alpha_1(z)-\beta(z)dz}ds$ is bounded above uniformly as $t\rightarrow\infty$.  Consequently the only possibility for the long term dynamics of $I(t)$ is that $I(t)\xrightarrow[t\rightarrow\infty]{}-\frac{a_1}{\lambda_1}$, and so $\rho_1(t)\xrightarrow[t\rightarrow\infty]{}-\frac{\lambda_1}{a_1}.$

We now obtain the final converge for $\tilde{u}_1(\cdot,t)$. From Lemma \ref{lma:AlphaExpDecay} we have that $u_1(x,t)=\tilde{u}_1(x,t)e^{-a_{1}\int_{0}^t\sum_{j=1}^N\rho_j(s)-\lambda_1t}\xrightarrow[t\rightarrow\infty]{}K_2p_1(x)$ for some constant $K_2>0$. By writing $\Sigma(x,t):=u_1(x,t)-K_2p_1(x)$, we have $\frac{K_2p_1(x)+\Sigma(x,t)}{K_2+\int_{\Omega_1}\Sigma(x,t)dx}=p_1(x)+\Sigma_1(x,t)$ where $\Sigma_1(x,t)\xrightarrow[t\rightarrow\infty]{}0$ uniformly in $x$ and exponentially fast.

\end{proof}

\subsection{Proof of Theorem \ref{thm:Global_attractors_UBD}}

The proof of Theorem 2.2 follows the same structure as the proof of Theorem 2.1, relying on analogous lemmas. The first is analogous to Lemma \ref{lma:AlphaExpDecay}.

\begin{lemma}\label{lma:AlphaExpDecay_UBD}
    Where $\tilde{u}=(\tilde{u}_1,...,\tilde{u}_N)$ is the solution to \eqref{eqn:Decoupled_UBD} and $\tilde{\rho}_i=\int_{\Omega_i}\tilde{u}_i(x,t)dx$, then there exists $A,B,K>0$ such that $\Vert{}\tilde{u}_i(\cdot,t)-Kp_i\Vert_{L^\infty(\mathbb{R})},\left|\frac{d\tilde{\rho}_i}{dt}\right|,\alpha_i(t)<e^{A-Bt}$, where $\alpha_i(t)=-\frac{d\tilde{\rho}_i(t)}{dt}\frac{1}{\tilde{\rho}_i(t)}$.
\end{lemma}

\begin{proof}[Proof of Lemma \ref{lma:AlphaExpDecay_UBD}]

The exponential separation results contained in Theorems 2.1 and 2.2 in \cite{huska2008exponential} allow us to show that \[\Vert{}\tilde{u}_i(\cdot,t)-Kp_i\Vert_{L^\infty(\mathbb{R})}\xrightarrow[t\rightarrow\infty]{}0,\]
     where $K>0$.  The details of this step can are contained in Lemmas 5 and 6 in \cite{adaptivedynamicsdivergingfitness}. From here, the procedure is identical to the proof of Lemma \ref{lma:AlphaExpDecay} where we use the solution $p_i$ $i=1,...,N$ to \eqref{eqn:eigenfunctions} in place of the solution to \eqref{eqn:eigenfunctionsDomain}.

\end{proof}

\begin{proof}[Proof of Lemma \ref{lma:MassEquivalent_UBD}]
    This is identical to the proof of Lemma \ref{lma:MassEquivalent}, where we use Lemma \ref{lma:AlphaExpDecay_UBD} in place of Lemma \ref{lma:AlphaExpDecay}.
\end{proof}

\begin{proof}[Proof of Theorem \ref{thm:Global_attractors_UBD}]
The proof of Theorem \ref{thm:Global_attractors_UBD} then follows identically to the proof of Theorem \ref{thm:Global_attractors_BD}. Specifically, Lemma \ref{lma:MassEquivalent_UBD} 
establishes the global existence of classical solutions to \eqref{eqn:ShiftingEnvironment}. Integrating \eqref{eqn:ShiftingEnvironment} with respect to $x$ to yields the
\begin{equation}\label{eqn:PopDynFull}
    \begin{cases}
        \frac{d\rho_i}{dt}=\rho_i\left(-\lambda_i+\alpha_i(t)-\sum_{i=1}^{N}a_{ij}\rho_j\right),\\
        \rho_i(0)=\int_{\mathbb{R}}u_{i,0}(y)dy.
    \end{cases}
\end{equation}
Since
$|\alpha_i(t)|\xrightarrow[t\rightarrow\infty]{}0$ for $i\in\{1,..,N\}$ this equation is asymptotically autonomous with the same dynamics as \eqref{eqn:GLV_UBD}. {Therefore (G) is satisfied in cases a)-c) as before since this result depends only on the structure of the dynamical system \eqref{eqn:PopDynFull} for $\rho$ and the exponential decay of $\alpha$ (which is obtained in Lemma \ref{lma:AlphaExpDecay_UBD}) .}
\end{proof}
\section{Trajectorial properties of the non-autonomous LV dynamics}\label{sec:RemarksDynamics}
As has been shown in the proof of Theorem \ref{thm:Global_attractors_UBD}, the long-time behaviour of the non-autonomous Lotka-Volterra system \eqref{eqn:PopDynFull} is the same as that of the autonomous one \eqref{eqn:GLV_UBD}, which plays a crucial role in the analysis of our models.  In this section we further compare trajectorial properties of the two dynamics, which is of its own interest. To this end, we consider the  autonomous unperturbed system
\begin{equation}\label{eqn:PopDynFull_3}
        \frac{d\widetilde{\rho}_i}{dt}=\widetilde{\rho}_i\left(-\lambda_i-\sum_{i=1}^{N}a_{ij}\widetilde{\rho}_j\right),
   \end{equation}
and its non-autonomous perturbation
\begin{equation}\label{eqn:PopDynFull_2}
      \frac{d\rho_i}{dt}=\rho_i\left(-\lambda_i-\sum_{i=1}^{N}a_{ij}\rho_j\right)+\alpha_i(t)\rho_i,
\end{equation}
\begin{figure}
    \centering
    \includegraphics[width=0.5\linewidth]{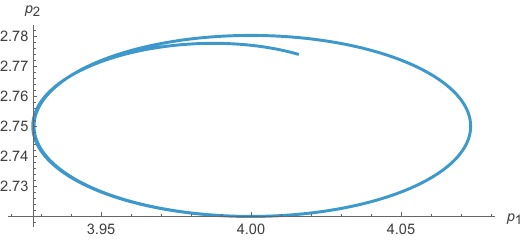}
    \caption{Trajectory of perturbed dynamics}
    \label{fig:two exxponents}
\end{figure}
Due to the choice of $\alpha(t)\to 0$, it is clear that the bounded trajectories of \eqref{eqn:PopDynFull_3} will have the  solutions of \eqref{eqn:PopDynFull_2} as $\omega$-limit sets. However, the most ``sensitive" properties of the dynamics will be lost. 

\begin{lemma}
    The system \eqref{eqn:PopDynFull_2} will have neither periodic trajectories nor resting points.
\end{lemma}
\begin{proof}
The first statement is a consequence of $\alpha_i(t)$ not being periodic. This is because in order for $\rho_i(t+T) = \rho_i(t)$, we must have  $\dot{\rho}_i(t+T) = {\dot{\rho}_i}(t)$. But the right hand side  of \eqref{eqn:PopDynFull_2} will be $\rho_i(-\lambda_i - \sum\limits_{i=1}^N a_{ij}\rho_j + \alpha_i(t+T))$, which cannot be  equal to the right hand side since $\alpha(t+T)\ne \alpha(t)$.

    The second part is straightforward: if all $\rho_i(t)$ are constant, the equations \eqref{eqn:PopDynFull_2} cannot hold. 
\end{proof}
We start the comparison between the behaviour of the trajectories of \eqref{eqn:PopDynFull_3} and of the unperturbed trajectories by two illustrative examples.
\begin{example}
Let us examine the following autonomous system of equations:
\begin{equation}
\label{eq: example 1}
    \begin{cases}
      \dot{\widetilde{\rho}_1} = \widetilde{\rho}_1(1.1 - 0.4\widetilde{\rho}_2)\\
      \dot{\widetilde{\rho}_2} = \widetilde{\rho}_2(0.1 \widetilde{\rho}_1-0.4).
    \end{cases}
\end{equation}
It has a unique equilibrium point in the interior of the first quadrant : $\widetilde{\rho}_1 =4, \widetilde{\rho}_2= 2.75 $, which is a centre. Therefore, all trajectories in the first quadrant are periodic. Additionally, the system has the  conserved quantity: $H(\widetilde{\rho}_1, \widetilde{\rho}_2) = 0.4\log \widetilde{\rho}_1 -0.1 \widetilde{\rho}_1 + 1.1 \log \widetilde{\rho}_2  - 0.4 \widetilde{\rho}_2$. The trajectories are the constant level sets of the invariant function. 

Consider an  arbitrary perturbation of \eqref{eq: example 1} by small $(\alpha_1(t), \alpha_2(t))$ that tend to 0 as $t\to+\infty$. 
\begin{equation}
\label{eq: example 12}
    \begin{cases}
      \dot{\rho_1} = \rho_1(\alpha_1(t)+1.1 - 0.4\rho_2)\\
      \dot{\rho_2} = \rho_2(\alpha_2(t)+0.1 \rho_1-0.4).
    \end{cases}
\end{equation}
Then by taking the directional derivative, the change in $H$ along the flow of \eqref{eq: example 1} is given by
\[\alpha_1 (0.4\, -0.1 \rho_1)+\alpha_2 (1.1\, -0.4 \rho_2)\ne 0.\] However, the difference, being proportional to $\alpha_1, \alpha_2$, tends to 0 as $t\to \infty.$ 

For numerical computations and illustration, we consider a concrete example with $\alpha_1 = e^{-t},\alpha_2= e^{-2t}$. An arbitrary trajectory of \eqref{eq: example 12} with such a perturbation  is depicted in Figure \ref{fig:two exxponents}. Clearly, the non-periodic trajectories of the perturbed system (\ref{eq: example 12}) have the periodic trajectories of \eqref{eq: example 1} as $\omega-$limit sets. 

Consider two trajectories of \eqref{eq: example 1} and \eqref{eq: example 12} with the same perturbation as above, i.e. $\alpha_1 = e^{-t},\alpha_2= e^{-2t}$. To make the initial perturbation small, we set the starting  time $t=50$ and assume that the curves start close to each other: $\widetilde{\rho}_1(50) = 10,\widetilde{\rho}_2(50) = 10, \rho_1(50) = 10.00001, \tilde{\rho}_2(50) = 10.00001$. Therefore,  when $t=50$, the perturbation has order  $10^{-22}$ and the difference in the initial conditions  $10^{-5}$.  

However, the trajectories do not remain this close for all values of $t$. Figure \ref{fig: norm of the distance} depicts the difference $||\widetilde{\rho} - \rho||(t)$ (here, the expression $||\cdot||$ is the standard Euclidean vector norm). Note that even for a relatively small timescale the distance between the two trajectories at the same time is a few orders of magnitude larger than the initial one.

\begin{figure}
\centering
    \subfigure[Norm of the difference between the solutions of \eqref{eq: example 1} and \eqref{eq: example 12} on a smaller timescale]{\includegraphics[scale=.6]{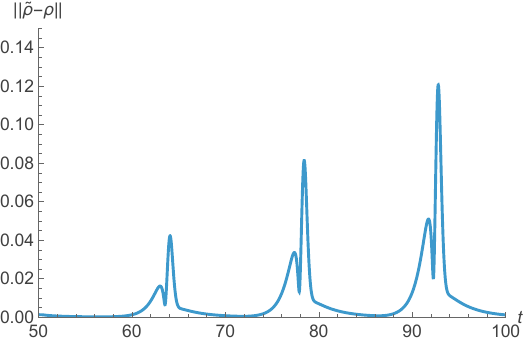}}
    \subfigure[Norm of the difference between the solutions of \eqref{eq: example 1} and \eqref{eq: example 12} on a larger  timescale]{\includegraphics[scale=.6]{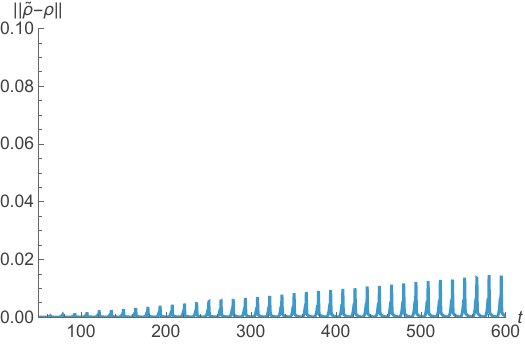}}
    \caption{Norm of the distance between perturbed and unperturbed trajectories. }
\label{fig: norm of the distance}
\end{figure}
    
\begin{figure}
    \centering
    \includegraphics[scale =.3]{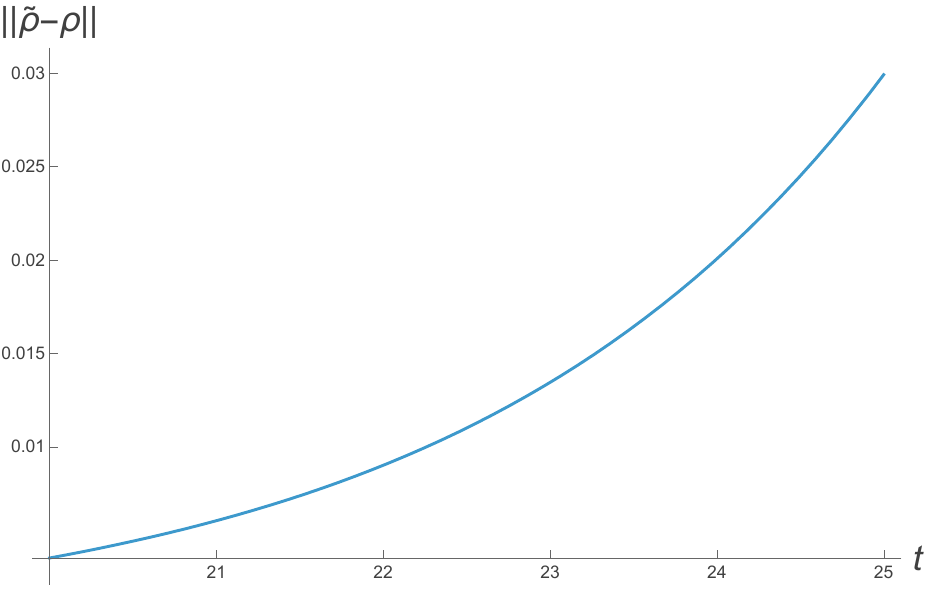}
    \caption{The norm of difference  $||\widetilde{\rho} - \rho||$ from Example \ref{ex:2}.}
    \label{fig:diffnumber2}
\end{figure}
\end{example}
\begin{example}
\label{ex:2}
To further the point of the previous example, consider the following two systems:
    \begin{equation}
\begin{cases}\dot{\widetilde{\rho}}_1=\widetilde{\rho}_1 (1.1\, -0.4 \widetilde{\rho}_2),\\\dot{\widetilde{\rho}}_2=\widetilde{\rho}_2 (0.4\, -0.1 \widetilde{\rho}_1),\\
\widetilde{\rho}_1(10)=6,\\
\widetilde{\rho}_2(10)=4
        \end{cases},\ \ \ \ \ \begin{cases}
            \dot{\rho}_1=\rho_1 (e^{-t}-0.4 \rho_2+1.1),\\\dot{\rho}_2=\rho_2 (e^{-2 t}-0.1 \rho_1+0.4),\\\rho_1(10)=6.00001,\\ \rho_2(10)=4.00001
        \end{cases}
    \end{equation}
    The plot of the norm of difference beteen the two trajectories is presented in Figure \ref{fig:diffnumber2} -- observe that it is an increasing function. 
\end{example}

The two examples above demonstrate that  the notion of ``closeness`` of the trajectories is defined as the closeness of the two sets: \cite{rasmussen2008bifurcations} establishes that for the same value of $t$ the two trajectories exhibit  different growth speeds. However, for finite time this difference can be levelled by choosing an appropriate shift in time $\tau$ for the perturbed system. Below we state an  altered version of the main theorem of \cite{rasmussen2008bifurcations} in which we consider $t\rightarrow +\infty$ instead of $t\to-\infty$ as in the original notation.
\begin{theorem}[\cite{rasmussen2008bifurcations}]

  Consider a non-autonomous differential equation $\dot{x} = f(t,x)$ with $t\in \mathbb{R}$  and $x\in D\subset \mathbb{R}^N$ with the flow $\lambda(t,\tau,x)$ (such that $
    \lambda(\tau,\tau,x)=x$) and an autonomous differential equation $\dot{x} = g(x)$ with the flow $\phi(t,x)$ with $\phi(0,x) = x$.  Assume that $\lim\limits_{t\to\infty}f(t,x) = g(x)$. Then the following holds: for all $T<0$ and $\epsilon>0$, there exists a $\tau_0>-T$ such that for all $T'<-T$ and $x\in K$ with $K$ compact and $\phi(t,x)\in K$ for all $t\in[0,T']$ we have
    \[
    ||\lambda(t - \tau, \tau,x) - \phi(t,x)||<\epsilon \ \ \forall \tau\ge \tau_0 \ \text{and} \ t\in[0,T'].
    \]
    
\end{theorem}
This theorem guarantees that one can find a time delay for every pair of trajectories of the perturbed and the unperturbed system that makes them stay in the $\epsilon$-vicinity of each other for some time. 

We will use the ideas behind the proof of this theorem to investigate the following: suppose that the two trajectories' starting points are  $\epsilon$ close; for how long can we guarantee that they stay close?

The main result of this section is the following lemma.
\begin{lemma}
\label{lem: dynamics LV}
    Assume that the trajectory of an unperturbed equation \eqref{eqn:PopDynFull_3} has a periodic cycle $C$ of width $diam (C)$  and that $M_1$ is the maximum value of $||\rho(t)||$ on this trajectory. Then, if the time delay $\tau$  was chosen to be sufficiently large, the time $T'$ that the perturbed trajectory will spend in the $2\epsilon$ vicinity of the unperturbed one will be $t\le \min\left\{\frac{\epsilon}{\mathrm{diam} C}, \frac{\epsilon}{M_1\max_{\xi\ge\tau}||\alpha(\xi)||}\right\}$.
\end{lemma}
\begin{proof}
For brevity, we stick to the notation of the theorem above, and denote the flow of the perturbed system by $\lambda(t,\tau,x)$ and the unperturbed by $\phi(t,x)$, with the left hand side of \eqref{eqn:PopDynFull} denoted by $f(t,x)$ and the right hand side of the unperturbed equation by $g(x)$.

When the $\omega$-limit set of the unperturbed trajectory is a point, it is clear that the two solutions will stay in the $\epsilon-$vicinity of each other indefinitely long, provided they started sufficiently close. In the case when the $\omega$-limit set is a cycle, things are a bit different. 

Using the definition of the flow, we can write the following estimations for the distance between the solutions:
\begin{equation*}
    \begin{split}
        ||\lambda(t + \tau, \tau,x) - \phi(t,x)||&\le \int\limits_0^t ||f(\lambda(s + \tau,s,x)) -g(\phi(s,x))||\,d s\\&\le\int\limits_0^t ||f(\lambda(s + \tau,s,x)) -g(\lambda(s + \tau,s,x))||\,d s \\&+ \int\limits_0^t ||g(\lambda(s + \tau,s,x)) -g(\phi(s,x))||\, d s.
    \end{split}
\end{equation*}
The first summand on the right hand side is proportional to $\alpha:$
\begin{equation*}
    \begin{split}
       \int\limits_0^t ||f(\lambda(s + \tau,\tau,x)) -g(\lambda(s + \tau,\tau,x))||d s & = \int\limits_0^t||\alpha(s + \tau)\rho(\lambda(s + \tau,\tau,x){)} ||\, d s.
    \end{split}
\end{equation*}
Here, the integrand is a vector with entries $(\alpha_i {\rho}_i)_i$ and the norm is the standard Euclidean vector norm.  Assume that $||{\rho}_i(t)||<M_1$ for some constant $M_1$ -- this follows from the $\omega$-limit set being the periodic trajectory. Then 
\[
\int\limits_0^t ||f(\lambda(s + \tau,s,x)) -g(\lambda(s + \tau,s,x))||d s \le M_1 t \max_{\xi\ge \tau}||\alpha(\xi)||. 
\]
For the second part, we have 
\begin{equation*}
    \begin{split}
       \int\limits_0^t ||g(\lambda(s + \tau,s,x)) -g(\phi(s,x))||d s&\le \mathrm{diam}\ C*t ,
    \end{split}
\end{equation*}
where $\mathrm{diam}\ C$ is the diameter of the periodic trajectory. 

Therefore, if we want the perturbed trajectory to stay in the $2\epsilon$-neighbourhood of the unperturbed one, it must hold that 
\[
t\le \min\left\{\frac{\epsilon}{\mathrm{diam} C}, \frac{\epsilon}{M_1\max_{\xi\ge\tau}||\alpha(\xi)||}\right\}.
\]
\end{proof}
\section*{Acknowledgements}
MHD was supported by EPSRC grant EP/Y008561/1. 



\printbibliography

\end{document}